\documentclass[preprint,12pt]{elsarticle} 

\usepackage{amssymb}
\usepackage{amsmath}

\usepackage{lineno}

\usepackage{stmaryrd} 
\usepackage{physics}
\usepackage{braket}
\usepackage{booktabs}
\usepackage{caption,subcaption}
\usepackage{hyperref}
\usepackage{cleveref}
\usepackage{algorithm}
\usepackage{algpseudocode}
\usepackage{arydshln}

\newtheorem{theorem}{Theorem}[section]

\newtheorem{assumption}{Assumption}[section]
\crefname{assumption}{assumption}{hypotheses}
\Crefname{assumption}{Assumption}{Hypotheses}

\newtheorem{lemma}{Lemma}[section]

\newtheorem{proposition}{Proposition}[section]
\crefname{proposition}{proposition}{propositions}
\Crefname{proposition}{Proposition}{Propositions}

\newdefinition{remark}{Remark}[section]
\crefname{remark}{remark}{remarks}
\Crefname{remark}{Remark}{Remarks}

\newproof{proof}{Proof}

\journal{}

\begin{document}

\begin{frontmatter}



\title{Optimal-rate error estimates and a twice decoupled solver for a backward Euler finite element scheme  of the Doyle-Fuller-Newman model of lithium-ion cells \tnoteref{t1} } 
 
\tnotetext[t1]{This work was funded by the National Natural Science Foundation of China (grant 12371437) and the Beijing Natural Science Foundation (grant Z240001).}

\author[label0,label1]{Shu Xu\corref{cor1}}
\ead{xushu@lsec.cc.ac.cn}
\author[label1,label2,label3]{Liqun Cao}
\ead{clq@lsec.cc.ac.cn}

\cortext[cor1]{Corresponding author}

\affiliation[label0]{organization={School of Mathematical Sciences, Peking University},
            city={Beijing},
            postcode={100871},
            country={China}}

\affiliation[label1]{organization={Institute of Computational Mathematics and Scientific/Engineering Computing, Academy of Mathematics and Systems Science, Chinese Academy of
Sciences},
            city={Beijing},
            postcode={100190},
            country={China}}

\affiliation[label2]{organization={State Key Laboratory of Scientific and Engineering Computing},
city={Beijing},
postcode={100190},
country={China}}

\affiliation[label3]{organization={National Center for Mathematics and Interdisciplinary Sciences, Chinese Academy of
Sciences},
city={Beijing},
postcode={100190},
country={China}}

\begin{abstract}
We investigate the convergence of a backward Euler finite element discretization applied to a multi-domain and multi-scale elliptic-parabolic system, derived from the Doyle-Fuller-Newman model for lithium-ion cells. We establish optimal-order error estimates for the solution in the norms $l^2(H^1)$ and $l^2(L^2(H^q_r))$, $q=0,1$. To improve computational efficiency, we propose a novel solver that  accelerates the solution process and controls memory usage. Numerical experiments with realistic battery parameters validate the theoretical error rates and  demonstrate the significantly superior performance of the proposed solver over existing solvers.

\end{abstract}

\begin{graphicalabstract}
\end{graphicalabstract}

\begin{highlights}
\item Comprehensive numerical analysis and implementation of the most widely used multiscale and multiphysics model for lithium-ion cells.
\item Optimal-order error estimates in the norms $l^2\qty(H^1)$ and $l^2\qty(L^2\qty(H^q_r)),\,q=0,1$, are rigorously established.
\item Numerical convergence rates in both 2D and 3D are validated, which is unaccessible in the literature. 
\item A novel solver is designed to accelerate computation and reduce memory overhead, effectively handling strong nonlinearity and multiscale complexity. 
\item Three-dimensional numerical experiments with real battery parameters demonstrate that the proposed solver is the fastest and flexible to balance performance.
\end{highlights}


\end{frontmatter}



\section{Introduction}\label{sec:intro}

    The Doyle-Fuller-Newman (DFN) model \cite{doyle_DFN_modeling_1993,fuller_simulation_1994}, commonly referred to as the pseudo-two-dimensional (P2D) model when the cell region is simplified to one dimension, is the most widely used physics-based model for lithium-ion cells.  It is essential in various engineering applications, including estimating the state of charge (SOC), analyzing capacity performance under diverse operating conditions, and generating impedance spectra \cite{smith_solid-state_2006,ramadesigan_modeling_2012,shi2015multi,plett_BMS1_2015,telmasre_impedance_2022,wu_physics-based_2024}. Furthermore, as a cornerstone of battery modeling, it provides a foundation for incorporating additional physics, such as temperature effects and mechanical stress, enabling various model extensions \cite{golmon_multiscale_2012,hariharan_mathematical_2018,chen_porous_2022,planella_continuum_2022}.  

    In this paper, we examine the following fully coupled nonlinear elliptic-parabolic system that characterizes the DFN model:
    \begin{equation}\label{eq:DFN_dfn_strong}
    \begin{cases}
    -\nabla \cdot\left(\kappa_1 \nabla \phi_1 - \kappa_2 \nabla f\qty(c_1)\right)=a_2 J, & (x, t) \in \Omega_1 \times\left(0, T\right), \\
    -\nabla \cdot\left(\sigma \nabla \phi_2\right)=-a_2 J, & (x, t) \in \Omega_2 \times\left(0, T\right), \\
    \varepsilon_1\frac{\partial c_1}{\partial t}-\nabla \cdot\left(k_1 \nabla c_1\right)=a_1 J, & (x, t) \in \Omega_1 \times\left(0, T\right), \\
    \frac{\partial c_2}{\partial t}-\frac{1}{r^2} \frac{\partial}{\partial r}\left(r^2 k_2 \frac{\partial c_2}{\partial r}\right)=0, & (x, r, t) \in \Omega_{2 r} \times\left(0, T \right),
    \end{cases}
    \end{equation} 
    where $\Omega_2\subset \Omega_1 \subset \mathbb{R}^N$ and $\Omega_{2 r} \subset \mathbb{R}^{N}\times \mathbb{R}_+$, $N=1,2,3$, represent multiply connected domains. 
    It is notable that the last equation does not involve $x$-differential operators, and $\Omega_{2 r}$ spans two distinct scales $x$ and $r$. Specifically, for each point $x \in \Omega_2$, there exists a singular parabolic equation defined over the $r$-coordinate. 

    \begin{figure}[htbp]
        \centering
        \includegraphics[width=0.6\textwidth]{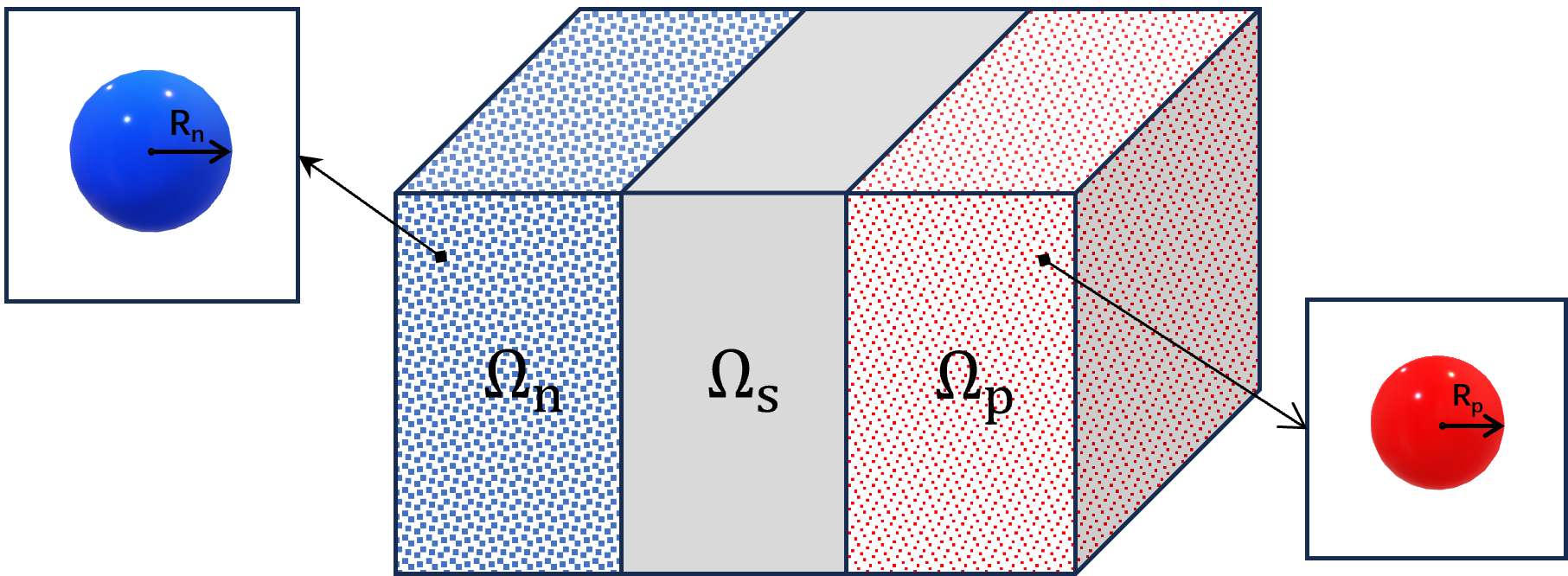} 
        \caption{A 3D schematic representation of a Li-ion cell.}
        \label{fig:domain} 
    \end{figure}

    In the DFN model, a lithium-ion cell occupies a domain $\Omega \subset \mathbb{R}^N$, where $1 \le N \le 3$, comprising three subdomains: the positive electrode $\Omega_\mathrm{p}$, the negative electrode $\Omega_\mathrm{n}$, and the separator  $\Omega_\mathrm{s}$. These regions typically form a laminated box structure, with $\bar \Omega = \bar \Omega_\mathrm{n} \cup \bar \Omega_\mathrm{s} \cup \bar \Omega_\mathrm{p}$, as illustrated in Figure~\ref{fig:domain}.  
    Following a macrohomogeneous approach, it is assumed that the electrode and electrolyte phases coexist within $\Omega_2$, and spherical particles with radii $R_\mathrm{p}$ and $R_\mathrm{n}$ are uniformly distributed throughout the electrodes, parameterized by the coordinates $\qty(x,r)\in \Omega_{2 r} := \qty(\Omega_\mathrm{n} \times (0, R_\mathrm{n})) \cup \qty(\Omega_\mathrm{p} \times (0, R_\mathrm{p}))$. We define the electrolyte region as $\Omega_1 = \Omega_\mathrm{n} \cup \Omega_\mathrm{s} \cup \Omega_\mathrm{p}$ and the electrode region as $\Omega_2 = \Omega_\mathrm{n} \cup \Omega_\mathrm{p}$.
    The model seeks to solve for four key variables: the electrolyte potential $\phi_1(x,t)$, the electrode potential $\phi_2(x,t)$, the lithium-ion concentration in the electrolyte  $c_1(x,t)$, and the lithium concentration within the particles $c_2(x,r,t)$.

    The multi-domain characteristics of the DFN model imply that the coefficients and nonlinear functions exhibit discontinuities across subdomains. Let $\boldsymbol{1}_{\Omega_m}$ denote the indicator function for the subdomain $\Omega_m$. It is assumed that  $\varepsilon_{1}$, $\sigma$, $k_{i}$ and $a_{i}$ ($i=1,2$), are piecewisely positive constants. When $v$ is piecewise constant, we denote its minimum and maximum values by $\underline{v} = \min_x v$ and $\bar{v} = \max_x{v}$, respectively. The coefficients $\kappa_i$, $i=1,2$ are defined as
    \begin{equation}\label{eq:def_kappa}
        \kappa_i = \sum_{m\in \set{\mathrm{n,s,p}}}\kappa_{im}\qty(c_1)\boldsymbol{1}_{\Omega_m},
    \end{equation}
    where  
    $\kappa_{1m},\kappa_{2m}\colon (0,+\infty) \rightarrow (0,+\infty)$. 
    For the spherical particles in the electrode regions, we set the radii as $R_\mathrm{s}  =\sum_{m\in \set{\mathrm{n,p}}}R_{m}\boldsymbol{1}_{\Omega_m}$ and the maximum lithium concentration as $c_{2,\max}  =\sum_{m\in \set{\text{n,p}}}c_{2,\max,m}\boldsymbol{1}_{\Omega_m}$.
    The source term $J$ depends on the lithium concentration at the particle surface, $\bar c_2(x,t) = c_2\qty(x,R_\mathrm{s}(x),t)$, and $\eta:= \phi_2 - \phi_1 - U$, where $U = \sum_{m\in \set{\mathrm{n,p}}}U_m(\bar c_2)\boldsymbol{1}_{\Omega_m}$, and $U_m \colon \qty(0,c_{2,\mathrm{max},m})  \rightarrow \mathbb{R}$. Then, $J$ is expressed as 
    \begin{equation}\label{eq:def_J}
        J =\sum_{m\in \set{\mathrm{n,p}}}J_{m}\qty(c_1,\bar c_2,\eta)\boldsymbol{1}_{\Omega_m},
    \end{equation}
    where 
    $J_m \colon (0,+\infty) \times \qty(0,c_{2,\mathrm{max},m}) \times \mathbb{R}  \rightarrow \mathbb{R}$. On the other hand, the function $f\colon (0,+\infty) \rightarrow \mathbb{R}$ is assumed to be independent of the spatial coordinates. The assumed domains of these functions imply the physical constraints $ c_1>0$ and $0 < c_2 < c_{2,\max}$. 
    
    To complete the system, we introduce 
    the initial conditions,
    \begin{gather*}\label{eq:strong_initial}
        c_1(x,0) = c_{10}(x) > 0, \quad x \in \Omega_1, \\ 
        c_2(x,r,0) = c_{20}(x,r),\quad 0 < c_{20}(x,r) < c_{2,\max}(x),\quad \qty(x,r) \in \Omega_{2r},
    \end{gather*}
    the boundary conditions with a given function $I\colon \Gamma  \rightarrow \mathbb{R}$ and a constant $F$,
    \begin{gather}
        -\left.\left(\kappa_1 \nabla \phi_1 - \kappa_2 \nabla f\qty(c_1)\right) \cdot \boldsymbol{n}\right|_{\partial \Omega}=0, \nonumber\\
        -\left.\sigma \nabla \phi_2 \cdot \boldsymbol{n}\right|_{\Gamma}=I,\quad -\left.\sigma \nabla \phi_2 \cdot \boldsymbol{n}\right|_{\partial \Omega_2 \setminus \Gamma}=0,\label{eq:strong_bdry_phi_2}\nonumber\\
        -\left.k_1 \nabla c_1 \cdot \boldsymbol{n}\right|_{\partial \Omega} = 0, \nonumber\\
        -r^2 k_2\left.\pdv{c_2(x)}{r}\right|_{r=0}= 0,\quad  -r^2 k_2\left.\pdv{c_2(x)}{r}\right|_{r=R_\mathrm{s}(x)}= \frac{J(x)}{F}, \quad x\in \Omega_2,\label{eq:strong_bdry_c_2}
    \end{gather}
    and the interface conditions 
    \begin{equation*}\label{eq:DFN_strong_interface}
        \begin{aligned}
        \left. \llbracket \phi_1 \rrbracket \right|_{\Gamma_{\mathrm{s}k}} &= 0, \quad \left. \left\llbracket \qty(\kappa_1\nabla \phi_1  -\kappa_2 \nabla f(c_1)) \cdot \boldsymbol{\nu} \right\rrbracket\right|_{\Gamma_{\mathrm{s}k}}= 0, \quad &k \in \qty{\mathrm{n,p}},\\ 
        \left. \llbracket c_1 \rrbracket \right|_{\Gamma_{\mathrm{s}k}} &= 0,\quad \left. \left\llbracket k_1 \nabla c_1 \cdot \boldsymbol{\nu} \right\rrbracket \right|_{\Gamma_{\mathrm{s}k}}= 0 , \quad &k \in \qty{\mathrm{n,p}},
        \end{aligned}   
    \end{equation*}
    where $\boldsymbol{n} $ denotes the outward unit normal vector, $\Gamma$ is a measurable subset of $\partial \Omega_2$ with positive measure,  $\Gamma_\mathrm{sn}$ and $\Gamma_\mathrm{sp}$ are the separator-negative and separator-positive interfaces respectively, $\left. \llbracket v \rrbracket \right|_{\Sigma}  $ represents the jump of a quantity $v$ across the interface $ \Sigma$, and $\boldsymbol{\nu} $ denotes the outward unit normal vector at $\partial \Omega_\mathrm{s}$.

    Note that the subsystem involving $\phi_1$ and $\phi_2$ forms a Neumann problem. For the existence of a solution, it is necessary that the following compatibility condition holds:
    \begin{equation}\label{eq:intro_compatibility_condition}
        \int_{\Omega_2} a_2 J(t) \, \mathrm{d}x = \int_{\Gamma}^{} I\qty(t)\, \dd s =0.
    \end{equation}
    Furthermore, to ensure the uniqueness of the solution, we impose the additional constraint:
    \begin{equation}\label{eq:intro_constraint_mean}
        \int_\Omega \phi_1 \, \dd x = 0.
    \end{equation}
    \begin{remark}
        Apart from making the coupling quasilinear by the coefficients \eqref{eq:def_kappa}, the coupling is influenced by the nonlinear function $J$ on the right-hand side of \eqref{eq:DFN_dfn_strong} and the Neumann boundary condition \eqref{eq:strong_bdry_c_2}, which represents the homogenization process and the electrochemical coupling between the electrode scale and the particle scale.
    \end{remark}

    More specifically, the DFN model is a particular case of the more general mathematical formulation \eqref{eq:DFN_dfn_strong}, when $\kappa_2 = \frac{2RT}{F}\kappa_1\qty(1-t^0_+)$, $f = \ln$, and $J$ is governed by the Butler-Volmer equation \cite{bermejo_implicit-explicit_2019}. 
    For a detailed derivation of the DFN model, we recommend referring to \cite{fuller_simulation_1994, ciucci_derivation_2011, arunachalam_veracity_2015, timms_asymptotic_pouch_2021, richardson_charge_2022}. As noted by \cite{xu2024DFNsemiFEM}, the change of variables for $\phi_1$ is avoided in this paper, ensuring that the finite element error analysis in this original formulation is also applicable to non-uniform temperature distributions.  

    \begin{assumption}\label{asp:IB_condition}
        $c_{10}\in L^2\qty(\Omega)$, 
        $c_{20} \in L^2\qty(\Omega_2 ; L^2_r \qty(0,R_\mathrm{s}(\cdot)))$.
        $c_{10}(x) > 0, \:x\in \Omega_1\; \mathrm{a.e.}$;
        $ 0 < c_{20}(x,r) < c_{2,\max}(x),\:(x,r) \in \Omega_{2r}\; \mathrm{a.e.}$.
        $I \in L^2\qty(0,T; L^2\qty(\Gamma))$ and satisfies the condition $\int_{\Gamma}^{} I\qty(t)\, \dd s =0$, $t\in \qty(0,T)\; \mathrm{a.e.}$.
    \end{assumption}

    \begin{assumption}\label{asp:nonlinear_function}
        $f \in C^1\qty(0,+\infty)$. For $m\in\set{\mathrm{n,p}}$,
        $\kappa_{1m},\kappa_{2m} \in C^2\qty(0,+\infty)$,  $ U_m \in C^2\qty(0,c_{2,\mathrm{max},m})$, $J_m \in C^1\qty((0,+\infty) \times \qty(0,c_{2,\mathrm{max},m})  \times \mathbb{R}  )$ and 
        \begin{displaymath}
            \exists \alpha >0\colon\: \forall \qty(c_1,\bar c_2, \eta) \in (0,+\infty) \times \qty(0,c_{2,\mathrm{max},m})  \times \mathbb{R}, \quad   \pdv{ J_m}{\eta}\qty(c_1,\bar c_2, \eta) \ge \alpha. 
        \end{displaymath} 
    \end{assumption}

    Given $T>0$ and setting $\Omega_{1T} = \Omega_1 \times \qty(0,T)$, $\Omega_{2rT} = \Omega_{2r} \times \qty(0,T)$. With function spaces and notations introduced in \ref{app_sec:func_space}, the weak solution of \eqref{eq:DFN_dfn_strong} can be defined as a quadruplet $\left(\phi_1, \phi_2, c_1, c_2\right)$, 
    \begin{gather}
        \notag \phi_1 \in L^2\left(0, T; H^1_*(\Omega) \right), \quad 
        \phi_2 \in L^2\left(0, T; H^1\left(\Omega_2\right)\right),  \\
        \notag c_1 \in C\left([0, T); L^{2}(\Omega)\right) \cap L^2\left(0, T; H^1(\Omega)\right), \quad \pdv{c_{1}}{t}  \in L^{2}\left(0, T; L^2(\Omega)\right),  \\
        \notag c_2 \in C\left([0, T); L^2\qty(\Omega_2;L_r^2\qty(0,R_\text{s}\qty(\cdot)))\right) \cap L^2\left(0, T; L^2\qty(\Omega_2;H_r^1\qty(0,R_\text{s}\qty(\cdot))) \right),  \\ 
        \notag \pdv{c_{2}}{t}  \in L^{2}\left(0, T;  L^2\qty(\Omega_2;L_r^2\qty(0,R_\text{s}(\cdot)))\right), 
    \end{gather}
    such that  $c_1(0) = c_{10}$, $c_2(0) = c_{20}$, $\kappa_1,\, \kappa_2 \in L^\infty\qty(\Omega_{1T})$, $f\qty(c_1) \in L^2\qty(0,T;H^1\qty(\Omega))$, $J \in L^2\qty(0,T;L^2\qty(\Omega_2))$,  $c_{1}(x,t)>0$, $(x,t) \in \Omega_{1T}\;\mathrm{a.e.}$, $ 0 < c_{2}(x,r,t) < c_{2,\max}(x)$, $(x,r,t)\in \Omega_{2rT}\;\mathrm{a.e.}$,
    and for $t \in(0, T)\text{ a.e.}$, 

    \begin{multline}\label{eq:DFN_weak_phi1}
        \int_{\Omega} \kappa_1\qty(t) \nabla \phi_1\qty(t) \cdot \nabla \varphi\, \dd x - \int_{\Omega} \kappa_2\qty(t) \nabla f\qty(c_1\qty(t)) \cdot \nabla \varphi \, \dd x -\int_{\Omega_2} a_2 J\qty(t) \varphi \, \dd x = 0,\\ 
        \forall \varphi \in H^{1}\left(\Omega\right),
    \end{multline}
    \begin{equation}\label{eq:DFN_weak_phi2_eqv}
        \int_{\Omega_2} \sigma \nabla \phi_2\qty(t) \cdot \nabla \varphi \, \dd x +\int_{\Omega_2}a_2 J\qty(t)\varphi \, \dd x + \int_{\Gamma}^{} I\qty(t) \varphi\, \dd s =0,\, \forall \varphi \in H^{1}\left(\Omega_2\right),
    \end{equation}
    \begin{equation}\label{eq:DFN_weak_c1}
        \int_{\Omega} \varepsilon_1 \dv{c_1}{t}\qty(t) \varphi \, \dd x +\int_{\Omega} k_1 \nabla c_1\qty(t) \cdot \nabla \varphi \, \dd x - \int_{\Omega_2} a_1 J\qty(t) \varphi \, \dd x =0, \, \forall \varphi \in H^{1}\left(\Omega\right),
    \end{equation}
    \begin{multline}\label{eq:DFN_weak_c2}
        \int_{\Omega_2} \int_0^{R_\text{s}(x)} \dv{c_2}{t}\qty(t) \psi r^2 \, \dd r\dd x + \int_{\Omega_2} \int_0^{R_\text{s}(x)} k_2 \frac{\partial c_2\qty(t)}{\partial r} \frac{\partial \psi}{\partial r} r^2 \;\dd r \dd x + \\   \int_{\Omega_2} \frac{R_\text{s}^2(x)}{F} J(x,t) \psi\left(x,R_\text{s}(x)\right) \dd x =0, \quad
        \forall \psi \in L^2\qty(\Omega_2;H_r^{1}\left(0, R_\text{s}(\cdot)\right)).
    \end{multline}

    \begin{remark}
        In certain contexts, it is more convenient to adopt an alternative formulation of \eqref{eq:DFN_weak_phi1}-\eqref{eq:DFN_weak_c2}, in particular by replacing \eqref{eq:DFN_weak_phi2_eqv} with the following variational equation:
        \begin{equation}\label{eq:DFN_weak_phi2}
            \int_{\Omega_2} \sigma \nabla \phi_2\qty(t) \cdot \nabla \varphi \, \dd x +\int_{\Omega_2}a_2 J\qty(t)\varphi \, \dd x + \int_{\Gamma}^{} I\qty(t) \varphi\, \dd s =0,\, \forall \varphi \in H_*^{1}\left(\Omega_2\right),
        \end{equation}
        where the test space coincides with the solution space. This formulation is particularly advantageous for establishing the existence of weak solutions  (see \cite[Proposition 2.16]{ramos_well-posedness_2016}). 
        The equivalence between the two formulations follows from the compatibility condition \eqref{eq:intro_compatibility_condition}.
        For consistency with the finite element semi-discrete error analysis developed in \cite{xu2024DFNsemiFEM}, we adopt the $H_*^1(\Omega)$ formulation \eqref{eq:DFN_weak_phi2}  throughout the remainder of this paper.
    \end{remark}
    
    Under a slightly more stringent version of Assumptions~\ref{asp:IB_condition} and \ref{asp:nonlinear_function} with higher regularity, local existence and uniqueness of the solution to the DFN model for $\dim \Omega = 1 $ was established by \cite{kroener_mathematical_2016,diaz_well-posedness_2019}. Moreover, the solution can be extended to a globally unique solution under specific conditions \cite{diaz_well-posedness_2019}. However, the problem of existence and uniqueness remains open for the case of $\dim \Omega = 2,3 $. The main mathematical challenges in this context include the strongly nonlinear source terms with singular behaviour, discontinuous and nonlinear coefficients, the lack of smoothness of the boundary and the pseudo-($N$+1)-dimensional structure.

    Efficient simulation of the DFN model continues to be an active area of research. 
    Various spatial discretization methods have been employed to convert the system of PDEs to a system of DAEs. These include the finite element method \cite{cai_lithium_2012,bermejo_implicit-explicit_2019,bermejo_numerical_2021,timms_asymptotic_pouch_2021}, finite difference method \cite{fuller_simulation_1994,newman_electrochemical_2019,mao_finitedifference_1994,nagarajan_mathematical_1998}, finite volume method \cite{zeng_efficient_2013,mazumder_faster-than-real-time_2013}, collocation method \cite{northrop_coordinate_2011,northrop_efficient_2015}, and hybrid method \cite{smith_solid-state_2006,kosch_computationally_2018}. 
    To achieve full discretization, computationally efficient time-stepping algorithms are applied, such as those detailed in \cite{brenan_numerical_1995,bermejo_implicit-explicit_2019,korotkin_dandeliion_2021,wickramanayake_novel_2024}. 
    Among these methods, the finite element method is particularly advantageous for handling irregular geometries, unconventional boundary conditions, heterogeneous compositions, and ensuring global conservation. As a result, it is increasingly adopted in battery modeling software \cite{korotkin_dandeliion_2021,aylagas_cidemod_2022,ai_jubat_2024,comsolBatteryGuide}. Additionally, the backward Euler method is commonly employed due to its unconditional stability \cite{mazumder_faster-than-real-time_2013,golmon_design_2014,han_fast_2021,aylagas_cidemod_2022,yin_batp2dfoam_2023,ai_improving_2023,kim_robust_2023}.  

    Despite notable advancements in numerical methods for the DFN model, research on rigorous error analysis remains limited \cite{bermejo_numerical_2021,xu2024DFNsemiFEM}. 
    The work in \cite{bermejo_numerical_2021} was the first to present a convergence analysis of the backward Euler finite element discretization of this model. However, the analysis has limitations:  it is inapplicable for the cases $N = 2,3$, where the imbedding $H^1\qty(D_2) \hookrightarrow C\qty(\bar D_2) $ fails and it requires additional assumptions or efforts to validate the interpolation operator $I_0^x$ used therein. Furthermore, the theoretical convergence rates presented do not align with observed numerical results. 
    Recently, a novel approach to estimating the error for $c_2$ was proposed in \cite{xu2024DFNsemiFEM}.  This approach utilizes a multiscale or pseudo-($N$+1)-dimensional projection, which aligns the analysis with the established framework from \cite{thomee_galerkin_2006}. Additionally, the inclusion of trace error estimation enabled proving optimal-order convergence concerning the radial mesh size.
    The primary aim of this paper is to extend the semi-discrete error analysis presented in \cite{xu2024DFNsemiFEM} to the fully discrete setting, addressing the associated gaps in the backward Euler finite element scheme. 

    An efficient solver is also crucial for enhancing simulation efficiency in practical applications, particularly when full discretization leads to a large system of nonlinear algebraic equations. 
    From a numerical algebra perspective, existing solvers can be categorized within the framework of the nonlinear Gauss-Seidel methods \cite{rheinboldt_iterative_1987,brune_composing_2015}, differing mainly in the way the problem is partitioned into blocks.
    Fully coupled solvers \cite{newman_electrochemical_2019,korotkin_dandeliion_2021,han_fast_2021} solve all unknown variables simultaneously using Newton's method. This approach typically converges rapidly  due to the comprehensive coupling of all equations, but it requires significant computational resources for solving the large-scale linear system at each Newton step.
    Decomposition solvers \cite{golmon_multiscale_2012,mazumder_faster-than-real-time_2013,latz_multiscale_2015,bermejo_implicit-explicit_2019,kim_robust_2023,yin_batp2dfoam_2023} decompose the problem into small block subproblems, introducing an outer iteration loop to handle the coupling between them.  Each subproblem may not be solved exactly, and an appropriate nonlinear solver may be used within an inner iteration loop.
    Generally, solvers with larger blocks, which couple more PDEs, tend to converge faster but increase storage and computational costs for solving subproblems. As emphasized in \cite{wu_newton-krylov-multigrid_2002}, understanding the problem's physics and the effect of coupling extent is crucial. 
    To accommodate varying computational capabilities, it is essential to strike a balance between solution speed and computational efficiency, motivating the development of innovative decomposition strategies tailored to the system's characteristics.

    The challenges in designing efficient solvers for the DFN model arise from the strong nonlinear reaction term and its multiscale nature. The reaction term $J$, governed by the Butler-Volmer equation \cite{latz_thermodynamic_2013}, involves complex mathematical functions such as exponentials and power roots, which hinder convergence and reduce robustness in fully decomposition solvers \cite{mazumder_faster-than-real-time_2013,kim_robust_2023,yin_batp2dfoam_2023}.  To address these difficulties, both fully coupled and tailored decomposition \cite{bermejo_implicit-explicit_2019} solvers  have been developed, offering improved convergence and robustness.  
    Furthermore, the multiscale pseudo-($N$+1)-dimensional equation introduces an additional dimension, significantly increasing computational scale and cost.  This motivates efforts to decouple $c_2$ in the outer \cite{golmon_multiscale_2012,latz_multiscale_2015,bermejo_implicit-explicit_2019} or inner iterations \cite{han_fast_2021}.
    Building on these challenges and insights, this paper also proposes a fast and robust solver that effectively controls memory usage.
    
    The primary contributions of this paper are threefold.  First, based on the finite element semi-discretization in \cite{xu2024DFNsemiFEM}, we propose the corresponding backward Euler finite element discretization and derive error estimates with optimal convergence rates for $N=1,2,3$. Second, we introduce a novel solver that fully decouples the microscopic variable $c_2$ through two effective decoupling procedures, along with an optional decomposition strategy. These approaches significantly accelerate the solution process while keeping memory usage low, making the solver particularly suitable for large-scale 2D/3D problems.  Finally, we validate the numerical convergence rates in 2D and 3D using real battery parameters and provide a comprehensive comparison with existing solvers. To the best of our knowledge, such numerical verification and comparison have not been previously reported in the literature.

    The paper is structured as follows. 
    \Cref{sec:fem_full} presents the backward Euler finite element discretization and derives error estimates with optimal convergence rates. 
    \Cref{sec:fem_solver} introduces the novel fast solver, detailing its key decoupling procedures and an optional decomposition strategy.  
    \Cref{sec:experiments} reports numerical experiments with real battery parameters to validate the error estimates and assess the solver's performance. 
    Finally, conclusions are drawn in \cref{sec:conclusions}. 
    
    \section{Error analysis of the finite element fully-discrete problems}
    \label{sec:fem_full}
        In this section, we only consider $\bar\Omega_{m}\subset \mathbb{R}^N,\; m\in\set{\mathrm{n, s, p}}$ as polygonal domains such that $\bar\Omega_{m}$ is the union of a finite number of polyhedra. Let $\mathcal{T}_{h,m}$ be a regular family of  triangulation for $\bar\Omega_{m},\; m\in\set{\mathrm{n, s, p}}$, and the meshes match at the interfaces, i.e., the points, edges and faces of discrete elements fully coincide at the interface. Hence, $\mathcal{T}_{h} := \bigcup_{m\in\set{\mathrm{n, s, p}}} \mathcal{T}_{h,m}$ and $\mathcal{T}_{2,h} := \bigcup_{m\in\set{\mathrm{n, p}}} \mathcal{T}_{h,m}$ are also regular family of triangulations for $\bar \Omega$ and $\bar \Omega_2$ respectively. Additionally, it is assumed that all $(K,P_K,\Sigma_K)$, $K\in\mathcal{T}_{h}$, are finite element affine families \citep{ciarlet2002finite}. For the closed interval $[0,R_\mathrm{s}(x)]$, $x\in \Omega_2$, let $\mathcal{T}_{\Delta_r}(x)$ be a family of regular meshes, such that $[0,R_\mathrm{s}(x)] = \bigcup_{I\in\mathcal{T}_{\Delta_r}(x)}I$. It is also assumed that each element of $\mathcal{T}_{\Delta_r}(x)$ is affine equivalent to a reference element. Since $R_\mathrm{s}$ is piecewise constant, here only two sets of meshes are considered, i.e.,
        \begin{equation*}
            \mathcal{T}_{\Delta_r}(x)= \left\{\begin{array}{l}\mathcal{T}_{\Delta_r \mathrm{p}}\qcomma{} x\in\Omega_\mathrm{p}, \\ \mathcal{T}_{\Delta_r \mathrm{n}}\qcomma{} x\in\Omega_\mathrm{n}.  \end{array} \right. 
        \end{equation*}
        For each element $K\in \mathcal{T}_{h}$ and $I \in \cup_{x\in \Omega_2} \mathcal{T}_{\Delta_r}(x)$, we use $h_K$ and $\Delta r_I$ for the diameter respectively. Let $h = \max_{K\in \mathcal{T}_{h}}  h_K $ and $\Delta r = \max_{I\in \cup_{x\in \Omega_2} \mathcal{T}_{\Delta_r}(x)}  \Delta r_I$.
    
        The unknowns $\phi_1$, $\phi_2$ and $c_1$ are discretized by piecewise-linear elements. Let $V_h^{(1)}\left(\bar \Omega\right)$ and $V_h^{(1)}\left(\bar \Omega_2\right)$ be the corresponding finite element spaces for $c_1$ and $\phi_2$, while 
        \begin{equation*}
        W_h\left(\bar \Omega\right)= \qty{w_h \in V_h^{(1)}\left(\bar \Omega\right):\: \int_{\Omega} w_h(x) \dd x = 0},
        \end{equation*}
        for $\phi_1$. 
        For the pseudo-($N$+1) dimensional $c_2$, piecewise-constant elements are used for discretization in the $x$ coordinate while piecewise-linear elements in the $r$ coordinate. Namely, a tensor product finite element space 
        \begin{equation}
        V_{h \Delta r}\left(\bar \Omega_{2 r}\right):=  \qty(V_h^{(0)}\left(\bar \Omega_\mathrm{n}\right) \otimes V_{\Delta r}^{(1)}\left[0,R_\mathrm{n}\right]) \bigcap \qty(V_h^{(0)}\left(\bar \Omega_\mathrm{p}\right) \otimes V_{\Delta r}^{(1)}\left[0,R_\mathrm{p}\right])
        \end{equation}
        is used.
      
      In order to obtain a full discretization, we consider a uniform mesh for the time variable $t$. Define $t_k:=k \Delta t, \, k=0,1, \ldots, K$, $\Delta t>0$ being the time-step, and $K:=[T / \Delta t]$, the integral part of $T / \Delta t$. By means of the backward Euler method, finite element fully discrete problems for  \eqref{eq:DFN_weak_phi1}-\eqref{eq:DFN_weak_c2} are proposed as follows:
    
      Given $\left(c_{1 h}^0, c_{2 h \Delta r}^0\right) \in V_h^{(1)}\left(\bar \Omega\right) \times V_{h \Delta r}\left(\bar \Omega_{2 r}\right)$, $c_{1 h}^0(x) > 0$, $0 < c_{2 h \Delta r}^0(x,r) < c_{2,\max}(x)$, find $\left(\phi_{1 h}^k, \phi_{2 h}^k, c_{1h}^k, c_{2 h\Delta r}^k\right)$  $\in W_h\left(\bar \Omega\right) \times V_h^{(1)}\left(\bar \Omega_2\right) \times V_h^{(1)}\left(\bar \Omega\right) \times V_{h \Delta r}\left(\bar \Omega_{2 r}\right)$, $k=1,\dots,K$, such that 
    \begin{equation}\label{eq:DFN_fem_phi1}
        \int_{\Omega} \kappa_{1h}^k\nabla \phi_{1h}^k \cdot \nabla w_h\, \dd x - \int_{\Omega} \kappa_{2h}^k\nabla f\qty(c_{1h}^k) \cdot \nabla w_h\,\dd x - \int_{\Omega_2} a_2 J_h^k w_h \,\dd x =0, \:\forall w_h \in W_h\left(\bar \Omega\right),
    \end{equation}
    \begin{equation}\label{eq:DFN_fem_phi2}
        \int_{\Omega_2} \sigma \nabla \phi_{2 h}^k \cdot \nabla v_{h}\, \dd x +\int_{\Omega_2} a_2 J_h^k v_h\, \dd x + \int_\Gamma I^k v_{h} \,\dd x =0,\; \forall v_h \in V_h^{(1)}\left(\bar \Omega_2\right),
    \end{equation}
    \begin{equation}\label{eq:DFN_fem_c1}
        \int_{ \Omega } \varepsilon_1 \frac{ c_{1 h}^k - c_{1 h}^{k-1}}{\tau} v_h\, \dd x +\int_{\Omega} k_{1} \nabla c_{1 h}^k \cdot \nabla v_h \, \dd x-\int_{\Omega_2} a_1 J_h^k v_h\, \dd x=0,\; \forall v_h \in V_h^{(1)}\left(\bar \Omega\right)
    \end{equation}
    \begin{multline}\label{eq:DFN_fem_c2}
        \int_{\Omega_2} \int_0^{R_\text{s}(x)} \frac{c_{2 h \Delta r}^k - c_{2 h \Delta r}^{k-1}}{\tau} v_{h \Delta r} r^2\, \dd r \dd x+\int_{\Omega_2} \int_0^{R_\text{s}(x)} k_{2} \frac{\partial c_{2 h \Delta r}^k}{\partial r} \frac{\partial v_{h \Delta r}}{\partial r} r^2 \, \dd r \dd x \\
        +\int_{\Omega_2} \frac{R_\text{s}^2(x)}{F} J_h^k(x) v_{h \Delta r}\left(x, R_\text{s}(x)\right)\, \dd x = 0, \quad \forall v_{h \Delta r} \in V_{h \Delta r}\left(\bar \Omega_{2 r}\right),
    \end{multline}
    where
    \begin{equation*}
        \kappa_{ih}^k = \sum_{m\in \set{\text{n,s,p}}}\kappa_{im}\qty(c_{1h}^k)\boldsymbol{1}_{\Omega_m}\qcomma{} i=1,2, \quad J_h^k =\sum_{m\in \set{\text{n,p}}}J_{m}\qty(c_{1h}^k,\bar c_{2h}^k,\eta_{h}^k)\boldsymbol{1}_{\Omega_m},
    \end{equation*}
    \begin{equation*}
        \bar{c}_{2 h}^k(x):=c_{2 h \Delta r}^k\left(x, R_\text{s}(x)\right), \quad \eta_{h}^k = \phi_{2h}^k - \phi_{1h}^k - U_h^k,\quad U_h^k = \sum_{m\in \set{\text{n,p}}} U_m(\bar c_{2h}^k)\boldsymbol{1}_{\Omega_m}.
    \end{equation*}

    \begin{remark}
        If the discrete compatibility condition 
        \begin{equation}\label{eq:fem_discrete_conservation}
            \int_{\Omega_2} a_2 J_h^k \, \dd x =  \int_{\Gamma} I^k \,\dd x  =0
        \end{equation}
        holds, which is the discrete analogue of \eqref{eq:intro_compatibility_condition}, then the finite element equation \eqref{eq:DFN_fem_phi1} is equivalent to
        \begin{equation*}\label{eq:DFN_fem_phi1_eqv}
            \int_{\Omega} \kappa_{1h}^k\nabla \phi_{1h}^k \cdot \nabla v_h\, \dd x - \int_{\Omega} \kappa_{2h}^k\nabla f\qty(c_{1h}^k) \cdot \nabla v_h\,\dd x - \int_{\Omega_2} a_2 J_h^k v_h \,\dd x =0, \:\forall v_h \in V_h^{(1)}\left(\bar \Omega\right).
        \end{equation*} 
        See \Cref{rm:exp_conservation} for practical considerations regarding the preservation of this condition under numerical quadrature.
    \end{remark}

    To present the finite element error analysis, we need to make the following assumptions.
    \begin{assumption}\label{asp:DFN_fem_IC}
        $$ c_{10} \in  H^2_\mathrm{pw}\qty(\Omega_1),\quad c_{20} \in  H^1\left(\Omega_2 ; H_r^1(0, R_\mathrm{s}\qty(\cdot))\right) \cap L^2\left(\Omega_2 ; H_r^2\qty(0, R_\mathrm{s}\qty(\cdot))\right).$$
    \end{assumption}
    \begin{assumption}\label{asp:DFN_fem_regularity}
        \begin{gather*}
            \phi_1 \in C\left(\qty[0, T] ; H_\mathrm{p w}^2\left(\Omega_1\right)\right), \phi_2 \in C\left(\qty[0, T] ; H_\mathrm{p w}^2\left(\Omega_2\right)\right),  
            c_1 \in H^1\left(0, T ; H_\mathrm{p w}^2\left(\Omega_1\right)\right), \\
            c_2 \in H^1\qty(0, T ; H^1\qty(\Omega_2 ; H_r^1(0, R_\text{s}\qty(\cdot))) \cap L^2\qty(\Omega_2 ; H_r^2\qty(0, R_\text{s}\qty(\cdot)))).
        \end{gather*}
    \end{assumption}
    \begin{assumption}\label{asp:DFN_prior_bound}
    There exist constants $L,M,N>0$, such that for $t\in \qty[0,T]$ a.e.,
    \begin{gather*}
        \left\| \phi_i(t)\right\|_{L^{\infty}\qty(\Omega_i)},\, \left\|\phi_{i h}(t)\right\|_{L^{\infty}\qty(\Omega_i)},\,
        \left\| \nabla \phi_i(t)\right\|_{L^{\infty}\qty(\Omega_i)},\,\norm{\nabla c_1(t)}_{L^{\infty}\qty(\Omega)} \leq L, \, i=1,2, \\ 
        \frac{1}{M} \leq c_1(t),\; c_{1h}(t) \leq M, \quad \frac{1}{N} \leq \frac{\bar c_2(t)}{c_{2,\max }},\; \frac{\bar c_{2h\Delta r}(t)}{c_{2,\max }} \leq\qty(1-\frac{1}{N}). \\ \label{eq:DFN_assumption_bound_c12}
    \end{gather*}
    \end{assumption}

    \subsection{Auxiliary results}
        To analyze the finite element error for the multiscale system, it is essential to introduce several projection operators into finite element spaces, particularly tensor product spaces, along with their approximation properties.

        Define the projection operator $P_{h}$ from $H^1(\Omega)$ to $V_h^{(1)}(\bar \Omega)$ such that $\forall u \in H^1(\Omega)$,
        \begin{equation}\label{eq:DFN_c1_projection_h}
            \int_{\Omega} k_1 \nabla\left(u-P_{h} u\right) \cdot \nabla \varphi_h \, \dd x+\int_{\Omega}  \left(u-P_{h} u\right) \varphi_h \, \dd x =0, \quad \forall \varphi_h \in V_h^{(1)}(\bar \Omega).
        \end{equation} 
        It is evident that $P_h$ is well-defined, with the following approximation error \cite{xu1982estimate}
        \begin{gather}\label{eq:DFN_c1_projection_h_error}
                \norm{u-P_{h} u}_{H^1\qty(\Omega)}  \le Ch\norm{u}_{H_\mathrm{pw}^2\qty(\Omega_1)}.
        \end{gather}
        Define the projection $P_{h\Delta r}: L^2\left(\Omega_2 ; H_r^1\left(0, R_\mathrm{s}\qty(\cdot)\right)\right) \rightarrow V_{h \Delta r}\left(\bar \Omega_{2 r}\right)$,
        \begin{multline}\label{eq:DFN_c2_projection_hr}
            \int_{\Omega_2} \int_0^{R_\mathrm{s}\qty(x)}\left(k_2 \frac{\partial\left(P_{h\Delta r} w-w\right)}{\partial r} \pdv{w_{h\Delta r}}{r} + \left(P_{h\Delta r} w-w\right) w_{h\Delta r}\right) r^2\, \dd r \dd x = 0, \\
            \forall w_{h\Delta r} \in V_{h \Delta r}\left(\bar \Omega_{2 r}\right) ,
        \end{multline} 
        and the projection error estimates \cite{xu2024DFNsemiFEM} are given as follows:
        \begin{lemma}\label{lemma:DFN_projection_hdr}
            If $w\in H^1\left(\Omega_2; H_r^1\left(0, R_\mathrm{s}\qty(\cdot)\right)\right) \cap L^2\left(\Omega_2 ; H_r^2\left(0, R_\mathrm{s}\qty(\cdot)\right)\right) $, then for $q=0,1$,
            \begin{equation*}\label{eq:DFN_projection_hdr_error}
                \norm{w-P_{h\Delta r}w}_{0,\Omega_2 ; q,r} \le C \qty(h\|w\|_{1,\Omega_2 ; 1,r} + \qty(\Delta r)^{2-q} \|w\|_{0,\Omega_2 ; 2,r}). 
            \end{equation*}
        \end{lemma}

        \begin{lemma}\label{lemma:DFN_projection_hdr_surf}
            If $w\in H^1\left(\Omega_2 ; H_r^1\left(0, R_\mathrm{s}\qty(\cdot)\right)\right) \cap L^2\left(\Omega_2 ; H_r^2\left(0, R_\mathrm{s}\qty(\cdot)\right)\right) $, we have
            \begin{equation*}\label{eq:DFN_projection_hdr_surf_error}
                \norm{\qty(w-P_{h\Delta r}w)\qty(\cdot,R_\mathrm{s}\qty(\cdot))}_{0,\Omega_2} \le C \qty(h\|w\|_{1,\Omega_2 ; 1,r} + \qty(\Delta r)^2 \|w\|_{0,\Omega_2 ; 2,r}).
            \end{equation*}
        \end{lemma}
        It can be seen from \cite{xu2024DFNsemiFEM} that \cref{lemma:DFN_projection_hdr_surf} is crucial for establishing the optimal-order error estimates with respect to $\Delta r$.
    
    \subsection{Error estimation for $c_1$ }  
    
        By virtue of the projection operator $P_h$, the finite element error can be decomposed as follows:
        \begin{equation}\label{app_eq:DFN_fem_c1_err_decomp}
            c_1\qty(t)-c_{1h}\qty(t)=\left(c_1\qty(t)-P_{h} c_1\qty(t)\right)+\left(P_{h} c_1\qty(t)-c_{1h}\qty(t)\right)=:\rho_{1}\qty(t) +\theta_{1}\qty(t).
        \end{equation}
        Since for $k=0,1$,
        \begin{equation} \label{eq:DFN_fem_c1_rho_err}
            \norm{\rho_{1}\qty(t)}_{k,\Omega} \le Ch\norm{c_1(t)}_{2,\Omega_1} \le Ch \qty(\norm{c_{10}}_{{2,\Omega_1}}+\int_0^t\norm{ \pdv{c_1}{t}\qty(s)}_{{2,\Omega_1}}\dd s),
        \end{equation}
        it remains to estimate $\theta_{1}\qty(t)$. 

        Next, since $\varepsilon_1$ is a piecewise positive constant function independent of time $t$, the norms $\norm{\varepsilon_1^{1/2}u}_{L^2\qty(\Omega)}$ and $\norm{u}_{L^2\qty(\Omega)}$ are equivalent. For simplicity, we assume $\varepsilon_1 \equiv 1$ in the subsequent analysis. 
        
        In addition, for a time-dependent variable $u(t)$, we use the notation $u^k = u\qty(t_k)$ for brevity.

    \begin{lemma}\label{lemma:DFN_fem_estimation_theta_c1}
        There exist an arbitrarily small number $\epsilon > 0$ and a positive constant $C$ that do not depend on $\tau$, $h$ and $\Delta r$, such that for each $k = 1,\dots,K$, 
        \begin{equation}\label{eq:DFN_fem_estimation_theta_c1}
            \begin{aligned}
                & \left\|\theta_1^k\right\|_{0,\Omega}^2-\left\|\theta_1^{k-1}\right\|_{0,\Omega}^2+\left\|\theta_1^k-\theta_1^{k-1}\right\|_{0,\Omega}^2+2 \underline{k_1} \tau\left\|\nabla \theta^k\right\|_{0,\Omega}^2 \\
                \leq & C h^2\left(\tau\norm{c_{10}}_{{2,\Omega_1}}^2+\tau\left\|\frac{\partial c_1}{\partial t}\right\|_{0; 2,\Omega_1}^2\right) + C h^2\left\|\frac{\partial c_1}{\partial t}\right\|_{0, k; 2,\Omega_1}^2  + C \tau^2\left\|\frac{\partial^2 c_1}{\partial t^2}\right\|_{0, k; 0,\Omega}^2 \\
                & + C\qty(\epsilon) \tau\left\|\theta_1^k\right\|_{0,\Omega}^2 + \epsilon \tau\qty(\left\|\bar c_2^k-\bar c_{2 h}^{k}\right\|_{0,\Omega_2}^2+\left\|\phi_1^k-\phi_{1h}^k\right\|_{0,\Omega}^2+\left\|\phi_2^k-\phi_{2 h}^k\right\|_{0,\Omega_2}^2).
            \end{aligned}
        \end{equation}
    \end{lemma}

    \begin{proof}
        The time derivative of $c_1$ at $t_k$ is 
        \begin{displaymath}
            \frac{\partial c_1}{\partial t}\qty(t_k) = \frac{c_1\qty(t_k)-c_1\left(t_{k-1}\right)}{\tau}+\frac{1}{\tau} \int_{t_{k-1}}^{t_k}\left(s-t_{k-1}\right) \frac{\partial^2 c_1}{\partial t^2}(s)\, \dd s.
        \end{displaymath}
        This leads to the decomposition of the temporal difference error as
        \begin{equation}\label{eq:DFN_tdiff_c1}
                \frac{\partial c_1}{\partial t}\qty(t_k)-\frac{c_{1 h}^k-c_{1h}^{k-1}}{\tau}
                =  \frac{\rho_1^k-\rho_1^{k-1}}{\tau}+ \frac{\theta_1^k-\theta_1^{k-1}}{\tau}+\int_{t_{k-1}}^{t_k}\left(\frac{s-t_{k-1}}{\tau}\right) \frac{\partial^2 c_1}{\partial t^2}(s)\, \dd s.
        \end{equation}
        Hence taking $\varphi= v_h$ in \eqref{eq:DFN_weak_c1} at $t=t_k$ and subtracting \eqref{eq:DFN_fem_c1}, it follows
        \begin{equation*}
            \begin{aligned}
                &\begin{aligned}
                & \int_{\Omega}\left(\theta_1^k-\theta_1^{k-1}\right) v_h\, \dd x +\tau \int_{\Omega} k_1 \nabla \theta_1^k \cdot \nabla v_h\, \dd x= \tau \int_{\Omega_2} a_1\left(J^k-J_h^k\right) v_h\, \dd x \dd t \\ 
                & - \int_{\Omega}\left(\rho_1^k-\rho_1^{k-1}\right) v_h\, \dd x - \int_{\Omega}\left(\int_{t_{k-1}}^{t_k}\left(s-t_{k-1}\right) \frac{\partial^2 c_1}{\partial t^2}(s)\, \dd s\right) v_h \, \dd x + \tau \int_{\Omega} \rho_1^k \cdot  v_h\, \dd x.
                \end{aligned}
            \end{aligned}
        \end{equation*}
        Next, by setting $v_h = \theta_1^k$, we have
        \begin{equation} \label{eq:DFN_theta1_estimation_raw}
            \begin{aligned}
                &\frac{1}{2}\left\|\theta_1^k\right\|_{0,\Omega}^2 - \frac{1}{2}\left\|\theta_1^{k-1}\right\|_{0,\Omega}^2 + \frac{1}{2}\left\|\theta_1^k-\theta_1^{k-1}\right\|_{0,\Omega}^2 +  \underline{k_1} \tau\left\|\nabla \theta_1^k\right\|_{0,\Omega}^2 \\ 
                \le & C\tau \left\|J^k-J_h^k\right\|_{0,\Omega_2} \norm{ \theta_1^k}_{0,\Omega}  +  C \left\|\rho_1^k-\rho_1^{k-1}\right\|_{0,\Omega} \norm{ \theta_1^k}_{0,\Omega} \\ 
                & + \tau^{\frac{3}{2}} \norm{\frac{\partial^2 c_1}{\partial t^2}}_{0,k;0,\Omega} \norm{\theta_1^k}_{0,\Omega}+ 
                \tau \norm{\rho_1^k}_{0,\Omega}\norm{\theta_1^k}_{0,\Omega}.
            \end{aligned}
        \end{equation}     
        Since Hypotheses~\ref{asp:nonlinear_function} and \ref{asp:DFN_prior_bound} hold, 
        \begin{equation} \label{eq:DFN_estimation_J}
            \begin{aligned}
                & \left\|J^k-J_h^k\right\|_{0,\Omega_2}  \le C\qty(\left\|c_1^k-c_{1h}^k\right\|_{0,\Omega}+\left\|\bar c_2^k - \bar c_{2h}^k\right\|_{0,\Omega_2}+\left\|\phi_1^k - \phi_{1h}^k\right\|_{0,\Omega} + \left\|\phi_2^k-\phi_{2 h}^k\right\|_{0,\Omega_2}) \\ 
                &\le  C \qty(\left\|\rho_1^k\right\|_{0,\Omega} + \norm{ \theta_1^k}_{0,\Omega} +\left\|\bar c_2^k - \bar c_{2h}^k\right\|_{0,\Omega_2}+\left\|\phi_1^k - \phi_{1h}^k\right\|_{0,\Omega} + \left\|\phi_2^k-\phi_{2 h}^k\right\|_{0,\Omega_2}).
            \end{aligned}
        \end{equation}
        We write 
        \begin{equation*}
            \rho_1^k-\rho_1^{k-1} = \qty(I-P_h)\qty(c_1^k -c_1^{k-1}) = \qty(I-P_h)\int_{t_{k-1}}^{t_k} \pdv{c_1}{t}\qty(s)\dd s,
        \end{equation*}
        and obtain
        \begin{equation*}
                \left\|\rho_1^k-\rho_1^{k-1}\right\|_{0,\Omega} 
                \le Ch\norm{ \int_{t_{k-1}}^{t_k} \pdv{c_1}{t}\qty(s)\dd s}_{2,\Omega_1} 
                \le Ch \tau^\frac{1}{2}\norm{ \pdv{c_1}{t}}_{0,k;2,\Omega_1}.
        \end{equation*}
        By virtue of \eqref{eq:DFN_fem_c1_rho_err}, 
        \begin{equation*}
            \norm{\rho_{1}^k}_{0,\Omega} \le Ch\norm{c_1(t_k)}_{2,\Omega_1} \le Ch \qty(\norm{c_{10}}_{{2,\Omega_1}}+\int_0^{t_k}\norm{ \pdv{c_1}{t}\qty(s)}_{{2,\Omega_1}}\dd s).
        \end{equation*}
        Together with the above estimates, Young's inequality implies that there exists $\epsilon > 0$, such that
        \begin{equation*}
            \begin{aligned}
                & \left\|\theta_1^k\right\|_{0,\Omega}^2-\left\|\theta_1^{k-1}\right\|_{0,\Omega}^2+\left\|\theta_1^k-\theta_1^{k-1}\right\|_{0,\Omega}^2+2 \underline{k_1} \tau\left\|\nabla \theta^k\right\|_{0,\Omega}^2 \\
                \leq & C h^2\left(\tau\norm{c_{10}}_{{2,\Omega_1}}^2+\tau\left\|\frac{\partial c_1}{\partial t}\right\|_{0; 2,\Omega_1}^2\right) + C h^2\left\|\frac{\partial c_1}{\partial t}\right\|_{0, k; 2,\Omega_1}^2  + C \tau^2\left\|\frac{\partial^2 c_1}{\partial t^2}\right\|_{0, k; 0,\Omega}^2 \\
                & + C\qty(\epsilon) \tau\left\|\theta_1^k\right\|_{0,\Omega}^2 + \epsilon \tau\qty(\left\|\bar c_2^k-\bar c_{2 h}^{k}\right\|_{0,\Omega_2}^2+\left\|\phi_1^k-\phi_{1h}^k\right\|_{0,\Omega}^2+\left\|\phi_2^k-\phi_{2 h}^k\right\|_{0,\Omega_2}^2).
            \end{aligned}
        \end{equation*}
    \end{proof}
    \subsection{Error estimation for $c_2$ }
    Thanks to the projection operator $P_{h\Delta r}$, the finite element error for $c_2$ can be similarly decomposed as follows:
    \begin{equation}\label{eq:DFN_fem_c2_err_decomp}
        c_2\qty(t)-c_{2h\Delta r}\qty(t)=\qty(c_2\qty(t)-P_{h\Delta r} c_2\qty(t))+\qty(P_{h\Delta r} c_2\qty(t) -c_{2h \Delta r}\qty(t))=:\rho_{2}\qty(t) +\theta_{2}\qty(t),
    \end{equation}
    and we have, setting $X=H^1\left(\Omega_2; H_r^1\left(0, R_\mathrm{s}\qty(\cdot)\right)\right) \cap L^2\left(\Omega_2 ; H_r^2\left(0, R_\mathrm{s}\qty(\cdot)\right)\right) $, 
    \begin{equation} \label{eq:DFN_fem_c2_rho_err}
        \begin{aligned}
            \norm{\rho_{2}\qty(t)}_{0,\Omega;q,r}  \le & Ch\norm{c_2(t)}_{1,\Omega_2;1,r} + C\qty(\Delta r)^{2-q}\norm{c_2(t)}_{0,\Omega_2;q,r} \\ 
            \le &C\qty(h + \qty(\Delta r)^{2-q} )\qty(\norm{c_{20}}_{X}+\int_0^t\norm{ \pdv{c_2}{t}\qty(s)}_{X}\dd s ),\quad q=0,1.
        \end{aligned}
    \end{equation}
    We proceed to give an estimate for $\theta_{2}\qty(t)$.
    \begin{lemma}\label{lemma:DFN_fem_estimation_theta_c2}
       There exists an arbitrarily small number $\epsilon \in \qty(0,1)$ that does not depend on $\tau$, $h$ and $\Delta r$, such that for each $k = 1,\dots,K$, 
        \begin{equation}\label{eq:DFN_fem_estimation_theta_c2}
            \begin{aligned}
            &\left\|\theta_2^k\right\|_{0,\Omega_2 ; 0, r}^2-\left\|\theta_2^{k-1}\right\|_{0,\Omega_2 ; 0, r}^2+\left\|\theta_2^k-\theta_2^{k-1}\right\|_{0,\Omega_2 ; 0, r}^2 + 2\underline{k_2} \tau\left\|\frac{\partial \theta_2^k}{\partial r}\right\|_{0,\Omega_2 ; 0, r}^2  \\
            \leq  &C \qty(h^2 + (\Delta r)^4)\qty(\tau\norm{c_{20}}_{X}^2+\tau\norm{ \pdv{c_2}{t}}_{{0;{X}}}^2 + \norm{ \pdv{c_2}{t}}^2_{0,k;X}) + C\tau^2 \left\|\frac{\partial^2 c_2}{\partial t^2}\right\|_{0, k; 0,\Omega_2;0,r}^2  \\
            + &C \epsilon \tau\qty(\left\|c_1^k-c_{1h}^k\right\|_{0,\Omega}^2+\left\|\phi_1^k-\phi_{1h}^k\right\|_{0,\Omega}^2+\left\|\phi_2^k-\phi_{2h}^k\right\|_{0,\Omega_2}^2)  
            + C\qty(\epsilon) \tau\left\|\theta_2^k\right\|_{0,\Omega_2 ;0, r}^2.
            \end{aligned}
        \end{equation}
    \end{lemma}

    \begin{proof}
        Similar to \eqref{eq:DFN_tdiff_c1}, the error of temporal difference can also be decomposed as 
        \begin{equation*}
            \frac{\partial c_2}{\partial t}\qty(t_k)-\frac{c_{2 h\Delta r}^k-c_{2h\Delta r}^{k-1}}{\tau}
            = \frac{\rho_2^k-\rho_2^{k-1}}{\tau} + \frac{\theta_2^k-\theta_2^{k-1}}{\tau}+\int_{t_{k-1}}^{t_k}\left(\frac{s-t_{k-1}}{\tau}\right) \frac{\partial^2 c_2}{\partial t^2}(s)\, \dd s.
        \end{equation*}
        Subtracting \eqref{eq:DFN_fem_c2} from \eqref{eq:DFN_weak_c2} at $t=t_k$, we obtain
        \begin{equation}\label{eq:DFN_theta2_estimation_raw}
            \begin{aligned}
            & \int_{\Omega_2} \int_0^{R_\mathrm{s}\qty(x)}\left(\theta_2^k-\theta_2^{k-1}\right) v_{h \Delta r} r^2 \,\dd r \dd x+\tau \int_{\Omega_2} \int_0^{R_\mathrm{s}(x)} k_2 \frac{\partial \theta_2^k}{\partial r} \frac{\partial v_{h \Delta r}}{\partial r} r^2 \,\dd r \dd x \\
            = & - \int_{\Omega_2}\int_0^{R_\mathrm{s}\qty(x)}\left(\int_{t_{k-1}}^{t_k}\left(s-t_{k-1}\right) \frac{\partial^2 c_2}{\partial t^2}(s)\, \dd s\right) v_{h \Delta r} r^2 \,\dd r\dd x \dd t \\
            & -\int_{\Omega_2} \int_0^{R_\mathrm{s}\qty(x)}\left(\rho_2^k-\rho_2^{k-1}\right) v_{h \Delta r} r^2 \,\dd r \dd x+\tau \int_{\Omega_2}\int_0^{R_\mathrm{s}(x)}  \rho_2^k v_{h \Delta r} r^2 \,\dd r \dd x  \\
            & -\tau \int_{\Omega_2} \frac{R_\mathrm{s}^2}{F}\left(J^k-J_h^k\right) \bar v_{h} \, \dd x, 
            \end{aligned}
        \end{equation}
        where $\bar v_{h} := v_{h \Delta r}\left(x, R_\text{s}(x)\right)$  and the penultimate term is due to \eqref{eq:DFN_c2_projection_hr}.

        We write 
        \begin{equation*}
            \rho_2^k-\rho_2^{k-1} = \qty(I-P_{h\Delta r})\qty(c_2^k -c_2^{k-1}) = \qty(I-P_{h\Delta r})\int_{t_{k-1}}^{t_k} \pdv{c_2}{t}\qty(s)\dd s,
        \end{equation*}
        and obtain
        \begin{equation*}
                \left\|\rho_2^k-\rho_2^{k-1}\right\|_{0,\Omega_2;0,r} 
                \le C\qty(h+\qty(\Delta r)^2) \tau^\frac{1}{2}\norm{ \pdv{c_2}{t}}_{0,k;X}.
        \end{equation*}
        For the last term, since $\norm{\bar{c}_2^k - \bar{c}_{2h}^k}_{0,\Omega_2} \le \norm{\bar \rho_2^k}_{0,\Omega_2} + \norm{\bar \theta_2^k}_{0,\Omega_2}$, it follows from \eqref{eq:DFN_estimation_J} that               
        \begin{multline}
            \left\|J^k-J_h^k\right\|_{0,\Omega_2} \le  C\left( \left\|c_1^k - c_{1h}^k\right\|_{0,\Omega} + \norm{\bar \rho_2^k}_{0,\Omega_2}  +\norm{\bar \theta_2^k}_{0,\Omega_2}\right. \\  \left.  +\left\|\phi_1^k - \phi_{1h}^k\right\|_{0,\Omega} + \left\|\phi_2^k-\phi_{2 h}^k\right\|_{0,\Omega_2}\right).
        \end{multline}
        By virtue of Lemma~\ref{lemma:DFN_projection_hdr_surf},
        \begin{equation} \label{eq:DFN_fem_c2_rho_surf}
            \begin{aligned}
                \norm{\bar \rho_2^k}_{0,\Omega_2}  \le &  Ch\norm{c_2(t_k)}_{1,\Omega_2;1,r} + C\qty(\Delta r)^2\norm{c_2(t_k)}_{0,\Omega_2;2,r} \\ 
                \le &C\qty(h + \qty(\Delta r)^2)\qty(\norm{c_{20}}_{X}+\int_0^{t_k}\norm{ \pdv{c_2}{t}\qty(s)}_{{{X}}}\dd s),
            \end{aligned}
        \end{equation}
        and due to Proposition~\ref{prop:DFN_radial_surface}, there is some $\epsilon_1 > 0$ to be determined later,
        \begin{equation}\label{eq:DFN_theta2_surface_estimation} 
            \left\|\bar \theta_{2}^k\right\|_{0,\Omega_2}^2 \leq \epsilon_1\left\|\frac{\partial \theta_{2}^k}{\partial r}\right\|_{0,\Omega_2 ; 0,r}^2 + C(\epsilon_1)\left\|\theta_{2}^k\right\|_{0,\Omega_2 ;0,r}^2.
        \end{equation}
        Hence, Young's inequality yields for any $\epsilon_2 \in \qty(0,1)$,
        \begin{equation}\label{eq:DFN_theta2_last_term_estimation} 
            \begin{aligned}
            \int_{t_{k-1}}^{t_k} \int_{\Omega_2} &\frac{R_\mathrm{s}^2}{F}\left(J^k-J_h^k\right) \bar v_{h} \, \dd x \dd s \le \epsilon_2 \tau \left\|J^k-J_h^k\right\|_{0,\Omega_2}^2  + C\qty(\epsilon_2) \tau \norm{\bar v_{h}}_{0,\Omega_2}^2  \\ 
            \le & C\tau\qty(h^2 + \qty(\Delta r)^4)\qty(\norm{c_{20}}_{X}^2+\norm{ \pdv{c_2}{t}}_{{0;{X}}}^2) +C\epsilon_1\tau\norm{\pdv{\theta_2^k}{r}}_{0,\Omega_2;0,r}^2 \\ 
            &+C\epsilon_2\tau\qty( \left\|c_1^k - c_{1h}^k\right\|_{0,\Omega}^2 + \left\|\phi_1^k - \phi_{1h}^k\right\|_{0,\Omega}^2 + \left\|\phi_2^k-\phi_{2 h}^k\right\|_{0,\Omega_2}^2)  \\ 
            &+ C\qty(\epsilon_1)\tau\norm{{\theta_2^k}}_{0,\Omega_2;0,r}^2  +  C\qty(\epsilon_2) \tau \norm{\bar v_{h}}_{0,\Omega_2}^2.
            \end{aligned}
        \end{equation}
        Finally, choosing $v_{h \Delta r}=\theta_{2}^k\in V_{h\Delta r}(\bar \Omega_{2r})$ in \eqref{eq:DFN_theta2_estimation_raw}, and applying the Cauchy-Schwarz inequality along with \eqref{eq:DFN_fem_c2_rho_err}, \eqref{eq:DFN_theta2_surface_estimation} and \eqref{eq:DFN_theta2_last_term_estimation}, we have
        \begin{displaymath}
            \begin{aligned}
            &\left\|\theta_2^k\right\|_{0,\Omega_2 ; 0, r}^2-\left\|\theta_2^{k-1}\right\|_{0,\Omega_2 ; 0, r}^2+\left\|\theta_2^k-\theta_2^{k-1}\right\|_{0,\Omega_2 ; 0, r}^2+2 \underline{k_2} \tau\left\|\frac{\partial \theta_2^k}{\partial r}\right\|_{0,\Omega_2 ; 0, r}^2 \\
            \leq & C\qty(h^2 + (\Delta r)^4)\norm{ \pdv{c_2}{t}}^2_{0,k;X}
            +C\tau \qty(h^2 + (\Delta r)^4)\qty(\norm{c_{20}}_{X}^2+\norm{ \pdv{c_2}{t}}_{{0;{X}}}^2)   \\ 
            &  + C \epsilon_2 \tau\qty(\left\|c_1^k-c_{1h}^k\right\|_{0,\Omega}^2+\left\|\phi_1^k-\phi_{1h}^k\right\|_{0,\Omega}^2+\left\|\phi_2^k-\phi_{2h}^k\right\|_{0,\Omega_2}^2) \\ 
            &+C\qty(\epsilon_1,\epsilon_2) \tau\left\|\theta_2^k\right\|_{0,\Omega_2 ;0, r}^2 + C\qty(\epsilon_2) \epsilon_1 \tau\left\|\frac{\partial \theta_2^k}{\partial r}\right\|_{0,\Omega_2 ; 0, r}^2 + C\tau^2\left\|\frac{\partial^2 c_2}{\partial t^2}\right\|_{0, k; 0,\Omega_2;0,r}^2 .
            \end{aligned}
        \end{displaymath}
        Estimation \eqref{eq:DFN_fem_estimation_theta_c2} follows by choosing $\epsilon_1$ sufficiently small.

    \end{proof}
    
    \subsection{Error estimation for the fully discrete problems}
        To this end, we still require error estimation for $\phi_1$ and $\phi_2$, which will then be integrated with the previously derived results to address the fully coupled problem.
        The error estimations for $\phi_1$ and $\phi_2$ have already been provided in \cite{xu2024DFNsemiFEM}, and we restate the results here for completeness:
        \begin{proposition}\label{prop:DFN_fem_estimation_Phi}
            There is a constant $C$ that does not depend on $\tau$, $h$ and $\Delta r$, such that
            \begin{multline*}
                \left\|\phi_1\qty(t_k)-\phi_{1 h}^k\right\|_{1,\Omega}^2 +\left\|\phi_2\qty(t_k)-\phi_{2 h}^k\right\|_{1,\Omega_2}^2  \\
                \le Ch^2\left(\left\|\phi_1\qty(t_k)\right\|_{2,\Omega_1}^2+\left\|\phi_2\qty(t_k)\right\|_{2,\Omega_2}^2\right)  \\ 
                + C\qty( \norm{c_1(t_k) - c_{1h}^k}_{1,\Omega}^2 + \norm{\bar c_2(t_k) - \bar c_{2h}^k}_{0,\Omega_2}^2 ).
            \end{multline*}
        \end{proposition}
        Now, we are ready to establish the convergence result of this paper.
    \begin{theorem}\label{thm:DFN_fem_full_estimation} 
        If Hypotheses~\ref{asp:nonlinear_function}-\ref{asp:DFN_prior_bound} hold, 
        there exists a constant $C$ that does not depend on $\tau$, $h$ and $\Delta r$ such that 
        \begin{gather}
            \begin{split}\label{eq:DFN_fem_full_estimation_macro}
                \sum_{k=1}^{K}\tau \qty(\left\|\phi_1^k-\phi_{1 h}^k\right\|_{1,\Omega}^2 + \left\|\phi_2^k-\phi_{2 h}^k\right\|_{1,\Omega_2}^2 + \left\|c_1^k-c_{1 h}^k\right\|_{1,\Omega}^2  + \left\|\bar c_2^k -\bar c_{2 h}^k\right\|_{0,\Omega_2}^2 ) \\
                \leq C\left(h^2+\tau^2+\qty(\Delta r)^4\right)  +C\left(\left\|c_{10}-c_{10,h}\right\|_{0,\Omega}^2+\left\|c_{20}-c_{20,h\Delta r}\right\|_{0,\Omega_2;0,r}^2\right)  
            \end{split} \\ 
            \begin{split}\label{eq:DFN_fem_full_estimation_micro}
                \sum_{k=1}^{K}\tau \left\|c_2^k-c_{2 h\Delta r}^k\right\|_{0,\Omega_2;q,r}^2& \leq C\left(h^2+\tau^2+\qty(\Delta r)^{4-2q}\right) \\  +C&\left(\left\|c_{10}-c_{10,h}\right\|_{0,\Omega}^2+\left\|c_{20}-c_{20,h\Delta r}\right\|_{0,\Omega_2;0,r}^2\right),\quad q=0,1. 
            \end{split}
        \end{gather}
    \end{theorem}

    \begin{proof}
        Combining Lemma~\ref{lemma:DFN_fem_estimation_theta_c1} with Proposition~\ref{prop:DFN_fem_estimation_Phi}, and utilizing \eqref{app_eq:DFN_fem_c1_err_decomp}, \eqref{eq:DFN_fem_c1_rho_err}, \eqref{eq:DFN_fem_c2_rho_surf} and \eqref{eq:DFN_theta2_surface_estimation}, there is a constant $\epsilon_1$ to be determined later, such that
        \begin{equation}\label{eq:DFN_fem_estimation_theta_c1_expand}
            \begin{aligned}
            & \left\|\theta_1^k\right\|_{0,\Omega}^2-\left\|\theta_1^{k-1}\right\|_{0,\Omega}^2+\left\|\theta_1^k-\theta_1^{k-1}\right\|_{0,\Omega}^2+2 \underline{k_1} \tau\left\|\nabla \theta^k\right\|_{0,\Omega}^2 \\
            \leq & C\qty(\epsilon_1) \tau\qty(\left\|\theta_1^k\right\|_{0,\Omega}^2+ \left\| \theta_2^k \right\|_{0,\Omega_2 ; 0, r}^2)+C \epsilon_1 \tau\qty(\left\|\nabla \theta_1^k\right\|_{0,\Omega}^2 + \left\|\frac{\partial \theta_2^k }{\partial r}\right\|_{0,\Omega_2 ; 0,r}^2 ) \\
            +&C  h^2\qty(\tau\left\|\phi_1^k\right\|_{2,\Omega_1}^2 + \tau\left\|\phi_2^k\right\|_{2,\Omega_2}^2 + \tau\norm{c_{10}}_{{2,\Omega_1}}^2 + \tau\norm{\frac{\partial c_1}{\partial t}}_{0; 2,\Omega_1}^2 +  \left\|\frac{\partial c_1}{\partial t}\right\|_{0, k; 2,\Omega_1}^2) \\
            + & C \tau^2\left\|\frac{\partial^2 c_1}{\partial t^2}\right\|_{0, k ; 0,\Omega}^2 
            +C\tau \qty(h^2+(\Delta r)^4)\qty(\norm{c_{20}}_{X}^2+\norm{ \pdv{c_2}{t}}_{{0;{X}}}^2).
            \end{aligned}
        \end{equation}
        Again, combining Lemma~\ref{lemma:DFN_fem_estimation_theta_c2} with Proposition~\ref{prop:DFN_fem_estimation_Phi}, together with \eqref{app_eq:DFN_fem_c1_err_decomp}, \eqref{eq:DFN_fem_c1_rho_err}, \eqref{eq:DFN_fem_c2_rho_surf} and \eqref{eq:DFN_theta2_surface_estimation}, there is also a constant $\epsilon_2\in \qty(0,1)$ to be determined later, such that
        \begin{equation}\label{eq:DFN_fem_estimation_theta_c2_expand}
            \begin{aligned}
                & \left\|\theta_2^k\right\|_{0,\Omega_2 ; 0, r}^2-\left\|\theta_2^{k-1}\right\|_{0,\Omega_2 ; 0, r}^2+\left\|\theta_2^k-\theta_2^{k-1}\right\|_{0,\Omega_2 ; 0, r}^2 + 2\underline{k_2} \tau\left\|\frac{\partial \theta_2^k}{\partial r}\right\|_{0,\Omega_2 ;0, r}^2 \\
                \leq &C\left(\epsilon_2\right) \tau\left(\left\|\theta_1^k\right\|_{0,\Omega}^2+\left\|\theta_2^k\right\|_{0,\Omega_2; 0, r}^2\right) + C\epsilon_2 \tau\left(\left\|\nabla \theta_1^k\right\|_{0,\Omega}^2 + \norm{\pdv{\theta_2^k}{r}}_{0,\Omega_2;0,r}^2\right) \\
                & +C \tau h^2\qty(\left\|\phi_1^k\right\|_{2,\Omega_1}^2 + \left\|\phi_2^k\right\|_{2,\Omega_2}^2 + \norm{c_{10}}_{{2,\Omega_1}}^2 + \norm{\frac{\partial c_1}{\partial t}}_{0; 2,\Omega_1}^2) \\ 
                &  +C\qty(h^2+(\Delta r)^4)\qty(\tau\norm{c_{20}}_{X}^2+\tau\norm{ \pdv{c_2}{t}}_{{0;{X}}}^2 + \norm{ \pdv{c_2}{t}}^2_{0,k;X}) \\
                &  + C\tau^2\left\|\frac{\partial^2 c_2}{\partial t^2}\right\|_{0, k; 0,\Omega_2;0,r}^2.
            \end{aligned}
        \end{equation}
        Next, adding \eqref{eq:DFN_fem_estimation_theta_c1_expand} and \eqref{eq:DFN_fem_estimation_theta_c2_expand}, and selecting $\epsilon_1$, $\epsilon_2$ sufficiently small, we have
        \begin{equation*}
            \begin{aligned}
            &\left\|\theta_1^k\right\|_{0,\Omega}^2+\left\|\theta_2^k\right\|_{0,\Omega_2 ; 0, r}^2  + \underline{k_1} \tau\left\|\nabla \theta_1^k\right\|_{0,\Omega}^2 + \underline{k_2} \tau\left\|\frac{\partial \theta_2^k}{\partial r}\right\|_{0,\Omega_2 ;0, r}^2 \\
            \leq & \left\|\theta_1^{k-1}\right\|_{0,\Omega}^2+\left\|\theta_2^{k-1}\right\|_{0,\Omega_2 ; 0, r}^2 + C\tau\qty( \left\|\theta_1^{k}\right\|_{0,\Omega}^2 + \left\|\theta_2^{k}\right\|_{0,\Omega_2 ; 0, r}^2) \\ 
            & +  C\tau^2\left(\left\|\frac{\partial^2 c_1}{\partial t^2}\right\|_{0, k ; 0,\Omega}^2 + \left\|\frac{\partial^2 c_2}{\partial t^2}\right\|_{0, k; 0, \Omega_2;0,r}^2 \right)  \\ 
            & +C \tau h^2\left(\left\|\phi_1\right\|_{\infty;2,\Omega_1}^2 + \left\|\phi_2\right\|_{\infty;2,\Omega_2}^2 + \norm{c_{10}}_{{2,\Omega_1}}^2  + \norm{c_{20}}_{X}^2\right)  \\
            & + C \tau h^2\qty(\norm{\frac{\partial c_1}{\partial t}}_{0; 2,\Omega_1}^2+\norm{ \pdv{c_2}{t}}_{{0;{X}}}^2)+ C\tau(\Delta r)^4\qty(\norm{c_{20}}_{X}^2+\norm{ \pdv{c_2}{t}}_{{0;{X}}}^2) \\ 
            & + C  h^2\left\|\frac{\partial c_1}{\partial t}\right\|_{0, k; 2,\Omega_1}^2 +C\qty(h^2+(\Delta r)^4)\norm{ \pdv{c_2}{t}}^2_{0,k;X}. 
            \end{aligned}
        \end{equation*}
        Discrete Gronwall's inequality yields 
        \begin{equation*}
            \begin{aligned}
                &\left\|\theta_1^k\right\|_{0,\Omega}^2   + \left\|\theta_2^k\right\|_{0,\Omega_2 ; 0, r}^2 +  \sum_{l=1}^k \underline{k_1} \tau\left\|\nabla \theta_1^l\right\|_{0,\Omega}^2 +  \sum_{l=1}^k \underline{k_2} \tau\left\|\frac{\partial \theta_2^l}{\partial r}\right\|_{0,\Omega_2 ; 0, r}^2  \\ 
                \leq &C\qty(\left\|\theta_1^{0}\right\|_{0,\Omega}^2 + \left\|\theta_2^{0}\right\|_{0,\Omega_2 ; 0, r}^2) +C(\Delta r)^4\qty(\norm{c_{20}}_{X}^2+\norm{ \pdv{c_2}{t}}_{{0;{X}}}^2)\\ 
                &+C h^2\left(\left\|\phi_1\right\|_{\infty;2,\Omega_1}^2 + \left\|\phi_2\right\|_{\infty;2,\Omega_2}^2 + \norm{c_{10}}_{{2,\Omega_1}}^2  + \norm{c_{20}}_{X}^2\right)\\ 
                & + C h^2\qty(\norm{\frac{\partial c_1}{\partial t}}_{0; 2,\Omega_1}^2+\norm{ \pdv{c_2}{t}}_{{0;{X}}}^2) 
                + C\tau^2\left(\left\|\frac{\partial^2 c_1}{\partial t^2}\right\|_{0; 0,\Omega}^2 + \left\|\frac{\partial^2 c_2}{\partial t^2}\right\|_{0; 0, \Omega_2;0,r}^2\right) \\ 
                \leq & C\qty(\left\|\theta_1^{0}\right\|_{0,\Omega}^2 + \left\|\theta_2^{0}\right\|_{0,\Omega_2 ; 0, r}^2 ) +C\left(\tau^2 + h^2+\qty(\Delta r)^4\right).
            \end{aligned}
        \end{equation*}
        Thus, we have the following $l^2\qty(H^1)$-norm estimation for $c_1$, demonstrating the optimal order of convergence,
        \begin{equation*}
            \begin{aligned}
            \sum_{l=1}^k \tau\left\|c_1^l-c_{1h}^{l}\right\|_{1,\Omega}^2 & \leq \sum_{l=1}^k 2 \tau\left(\left\|\rho_1^l\right\|_{1,\Omega}^2+ \left\|\theta_1^l\right\|_{0,\Omega}^2 +\left\|\nabla \theta_1^l\right\|_{0,\Omega}^2\right) \\
            & \leq C h^2 + C\tau \sum_{l=1}^k\qty(  \left\|\theta_1^l\right\|_{0,\Omega}^2 +  \left\|\nabla \theta_1^l\right\|_{0,\Omega}^2) \\
            & \leq C\left(\tau^2+h^2+(\Delta r)^4\right)+C\left(\left\|\theta_1^0\right\|_{0,\Omega}^2+\left\|\theta_2^0\right\|_{0,\Omega_2 ; 0, r}^2\right),
            \end{aligned}
        \end{equation*}
        and $l^2\qty(L^2\qty(H_r^q))$-norm, $q=0,1$, for $c_2$
        \begin{equation*}
            \begin{aligned}
            \sum_{l=1}^k \tau\left\|c_2^l-c_{2h}^l\right\|_{0,\Omega_2;q,r}^2  \leq &  \sum_{l=1}^k 2 \tau\left(\left\|\rho_2^l\right\|_{0,\Omega_2;q,r}^2+ \left\|\theta_2^l\right\|_{0,\Omega_2;q,r}^2\right)  \\ 
            \leq  & C\left(\tau^2 + h^2+(\Delta r)^{4-2q}\right)+C\left(\left\|\theta_1^0\right\|_{0,\Omega}^2+\left\|\theta_2^0\right\|_{0,\Omega_2;0,r}^2\right).
            \end{aligned}
        \end{equation*}
        Further with $l^2\qty(L^2)$-norm estimation for $\bar c_2$,
        \begin{equation*}
            \begin{aligned}
            \sum_{l=1}^k \tau\left\|\bar{c}_2^l-\bar{c}_{2h}^l\right\|_{0,\Omega_2}^2  \leq &  \sum_{l=1}^k C \tau\left(\left\|\bar \rho_2^l\right\|_{0,\Omega_2}^2+ \left\| \theta_2^l\right\|_{0,\Omega_2;0,r}^2 +  \left\| \pdv{\theta_2^l}{r}\right\|_{0,\Omega_2;0,r}^2\right) \\ 
            \leq  & C\left(\tau^2 + h^2+(\Delta r)^4\right)+C\left(\left\|\theta_1^0\right\|_{0,\Omega}^2+\left\|\theta_2^0\right\|_{0,\Omega_2;0,r}^2\right),
            \end{aligned}
        \end{equation*}
        we also have $l^2\qty(H^1)$-norm estimation for $\phi_1$ and $\phi_2$ with the optimal order of convergence,
        \begin{equation*}
            \begin{aligned}
            & \sum_{t=1}^k \tau\left\|\phi_1^l-\phi_1^l\right\|_{1,\Omega}^2+\tau\left\|\phi_2^l-\phi_2^l\right\|_{1,\Omega_2}^2 \\
            \leq & C h^2 \left(\left\|\phi_1\right\|_{\infty;2,\Omega_1}^2+\left\|\phi_2\right\|_{\infty;2,\Omega_2}^2\right)  +C \sum_{l=1}^k \qty(\tau\left\|c_1^l-c_{1h}^l\right\|_{1,\Omega}^2+\tau\left\|\bar{c}_2^l-\bar{c}_{2h}^{l}\right\|_{0,\Omega_2}^2) \\ 
            \leq & C\left(\tau^2 + h^2+(\Delta r)^4\right)+C\left(\left\|\theta_1^0\right\|_{0,\Omega}^2+\left\|\theta_2^0\right\|_{0,\Omega_2;0,r}^2\right).
            \end{aligned}
        \end{equation*}
        Finally, notice that triangle inequality yields
        \begin{gather*}
          \left\|\theta_{1}^0\right\|_{L^2\left(\Omega\right)} \le \norm{c_{10}-c_{1h}^0}_{L^2\left(\Omega\right)} + Ch, \\ 
          \left\|\theta_{2}^0\right\|_{L^2\qty(\Omega_2 ; L_r^2\qty(0, R_\mathrm{s}\qty(\cdot)))} \le \norm{c_{20}-c_{2h\Delta r}^0}_{L^2\qty(\Omega_2 ; L_r^2\qty(0, R_\mathrm{s}\qty(\cdot)))} + C\qty(h+\qty(\Delta r)^2),
        \end{gather*}
        which completes the proof.
    \end{proof}
    
    \section{A twice decoupled solver}
    \label{sec:fem_solver}
        
        Let \(\left(\phi_{1 h}^{k-1}, \phi_{2 h}^{k-1}, c_{1h}^{k-1}, c_{2 h\Delta r}^{k-1}\right)\) denote the fully discrete solution at the \((k-1)\)-th time step. In this section, we focus on efficiently solving the fully discrete system \eqref{eq:DFN_fem_phi1}--\eqref{eq:DFN_fem_c2} to obtain the next solution tuple \(\left(\phi_{1 h}^{k}, \phi_{2 h}^{k}, c_{1h}^{k}, c_{2 h\Delta r}^{k}\right)\) at time step \(k\).

        We propose a novel, twice-decoupled fast solver to achieve this goal. The effectiveness and computational advantages of this approach will be demonstrated through numerical experiments in Subsection~\ref{subsec:exp_solver}.

        \subsection{First decoupling: local inversion}
            Let \(\{\chi_i \psi_j\}_{i=1,\dots,M_2}^{j=1,\dots,N_i}\) denote the basis functions for \(V_{h \Delta r}(\bar \Omega_{2 r})\), where $\chi_i = \boldsymbol{1}_{\hat e_i}$ and $\set{\psi_j}_{j=1,\dots, N_{i}}$ is the nodal basis for $V^\qty(1)_{\Delta r}\qty[0,R_m]$, for each $ \hat e_i \in \mathcal{T}_{h,m}$,  $m\in \set{\text{n,p}}$. Then the solution $c_{2h\Delta r}^{k}\in V_{h \Delta r}\left(\bar \Omega_{2 r}\right)$ can be expressed as
            \begin{equation*}
            c_{2h\Delta r}^{k} (x,r) = \sum_{i=1}^{M_2}\sum_{j=1}^{N_{i}} c_{2h\Delta r,ij}^{k} \chi_i(x) \psi_j(r).
            \end{equation*}

            Substituting the test function \(v_{h \Delta r}(x,r) = \chi_i(x)\psi_{j}(r)\) into \eqref{eq:DFN_fem_c2}, we obtain for each \(\hat e_i \in \mathcal{T}_{2,h}\) the following equation: 
            \begin{multline}\label{eq:electrochemistry_nl_coup_c2}
                \int_0^{R_m} \frac{c_{2 h \Delta r,il}^{k} - c_{2 h \Delta r,il}^{k-1}}{\tau} \psi_{l} \psi_{j} r^2 \dd r + \int_0^{R_m} k_2 c_{2 h \Delta r,il}^{k} \frac{\partial \psi_l}{\partial r} \frac{\partial \psi_{j}}{\partial r} r^2 \dd r  \\
                + \frac{\delta_{jN_i}}{\qty|\hat e_i|}\int_{\hat e_i} \qty(\frac{R_m^2}{F}  J_{m}\qty(c_{1h}^{k},c_{2 h \Delta r,iN_i}^{k},\phi_{2h}^{k}-\phi_{1h}^{k}-U_m\qty(c_{2 h \Delta r,iN_i}^{k})))  \dd x = 0,
            \end{multline} 
            where the Einstein summation convention is applied to index \(l\).
            Equation~\eqref{eq:electrochemistry_nl_coup_c2} can be rewritten in matrix form as
            \begin{equation}\label{eq:electrochemistry_nl_coup_c2_matrix}
                AC_i^{k} - MC_i^{k-1} + J_{h,i}^{k}e_{N_i} = 0,
            \end{equation}
            where $C_i^{k} = \qty(c_{2 h \Delta r,i1}^{k}, \cdots, c_{2 h \Delta r,iN_i}^{k})^{\mathrm{T}}$, $C_i^{k-1} = \qty(c_{2 h \Delta r,i1}^{k-1}, \cdots, c_{2 h \Delta r,iN_i}^{k-1})^{\mathrm{T}}$, $e_{N_i} = \qty[0,\dots,0,1]^{\mathrm{T}}\in \mathbb{R}^{N_i}$, $A = M + \tau K \in  \mathbb{R}^{N_i\times N_i}$,
            \begin{equation*}
                M_{ij} =  \int_0^{R_m} \psi_{i} \psi_{j} r^2 \dd r,\quad K_{ij} =\int_0^{R_m} k_2 \frac{\partial \psi_i}{\partial r} \frac{\partial \psi_{j}}{\partial r} r^2 \dd r, 
            \end{equation*}
            and the nonlinear function
            \begin{equation*}
                J_{h,i}^{k}:= \frac{\tau}{\qty|\hat e_i|}\int_{\hat e_i} \qty(\frac{R_m^2}{F}  J_{m}\qty(c_{1h}^{k},c_{2 h \Delta r,iN_i}^{k},\phi_{2h}^{k}-\phi_{1h}^{k}-U_m\qty(c_{2 h \Delta r,iN_i}^{k})))  \dd x.
            \end{equation*}

            Although only the \(N_i\)-th equation in \eqref{eq:electrochemistry_nl_coup_c2_matrix} is nonlinear, it is coupled to the other linear equations via the matrix \(A\). Since \(A\) is symmetric and positive definite, the system is equivalently written as
            \begin{equation}\label{eq:electrochemistry_nl_coup_c2_matrix_equiv}
                C_i^{k} - A^{-1}MC_i^{k-1} + J_{h,i}^{k} A^{-1}e_{N_i} = 0.
            \end{equation}
            Taking the dot product with \(e_{N_i}^{\mathrm{T}}\), we isolate the scalar nonlinear equation associated with the surface DOF \(c_{2h\Delta r,iN_i}^{k}\): 
            \begin{equation}\label{eq:electrochemistry_nl_coup_c2_scalar}
                c_{2h\Delta r,iN_i}^{k} +  e_{N_i}^{\mathrm{T}}A^{-1}e_{N_i} J_{h,i}^{k} - e_{N_i}^{\mathrm{T}}A^{-1}MC_i^{k-1} = 0.
            \end{equation}
            Due to the low order and tridiagonal structure of $A$, the matrix inversion is computationally inexpensive and can be efficiently carried out using algorithms like Thomas algorithm (TDMA) or cholesky decomposition.
            
            This observation motivates a scale-decoupled solution strategy for the fully discrete system \eqref{eq:DFN_fem_phi1}--\eqref{eq:DFN_fem_c2}. First, we solve the decoupled system consisting of \eqref{eq:DFN_fem_phi1}--\eqref{eq:DFN_fem_c1} and the scalar equations \eqref{eq:electrochemistry_nl_coup_c2_scalar} to determine \(c_{1h}^{k}\), \(\phi_{1h}^{k}\), \(\phi_{2h}^{k}\), and the surface values \(c_{2h\Delta r,iN_i}^{k}\) for \(i = 1, \dots, M_2\). Then, using backward substitution into \eqref{eq:electrochemistry_nl_coup_c2_matrix}, we compute the remaining degrees of freedom in \(c_{2h\Delta r}^{k}\). 
            This approach is naturally parallelizable and significantly reduces the computational complexity by decoupling the spatial and particle-scale resolutions.

            We now propose the following novel solver to efficiently handle the scale-coupled nonlinear system: 


            \textbf{Step 1} (Decoupled global solve): Find $c_{1 h}^{k}\in V_h^{(1)}\left(\bar \Omega\right) $, $\phi_{1 h}^{k} \in  W_h\left(\bar \Omega\right)$, $\phi_{2 h}^{k}\in V_h^{(1)}\left(\bar \Omega_2\right)$,  and  $c_{2h\Delta r,iN_i}^{k}$, $ i=1,\dots, M_2$, such that for all $v_{1h} \in  V_h^{(1)}\left(\bar \Omega\right) $, $w_h \in  W_h\left(\bar \Omega\right) $, $ v_{2h} \in V_h^{(1)}\left(\bar \Omega_2\right) $,
            \begin{equation}\label{eq:electrochemistry_step1-1_c1}
                \int_{ \Omega} \varepsilon_1\frac{c_{1 h}^{k}-c_{1 h}^{k-1}}{\tau} v_{1h}  \, \dd x +\int_{\Omega} k_1 \nabla c_{1 h}^{k} \cdot \nabla v_{1h} \, \dd x -\int_{\Omega_2} a_1  J_h^{k} v_{1h} \, \dd x=0,
            \end{equation}
            \begin{equation}\label{eq:electrochemistry_step1-1_phi1}
                \int_{\Omega} \kappa_{1h}^{k}\nabla \phi_{1h}^k \cdot \nabla w_h\, \dd x - \int_{\Omega} \kappa_{2h}^{k}\nabla f\qty(c_{1h}^{k}) \cdot \nabla w_h\,\dd x - \int_{\Omega_2}  a_2  J_h^{k}  w_h \,\dd x =0,
            \end{equation} 
            \begin{equation}\label{eq:electrochemistry_step1-1_phi2} 
                \int_{\Omega_2} \sigma \nabla \phi_{2 h}^{k} \cdot \nabla v_{2h}\,\dd x + \int_{\Omega_2}a_2  J_h^{k} v_{2h}\, \dd x + \int_\Gamma I^k v_{2h} \,\dd x  =0, 
            \end{equation} 
            \begin{equation}\label{eq:electrochemistry_step1-1_c2_scalar}
                c_{2h\Delta r,iN_i}^{k} +  e_{N_i}^{\mathrm{T}}A^{-1}e_{N_i} J_{h,i}^{k} - e_{N_i}^{\mathrm{T}}A^{-1}MC_i^{k-1} = 0, \; i=1,\dots, M_2.
            \end{equation}

            \textbf{Step 2} (Local backward recovery):
            For each \(i = 1, \dots, M_2\), recover the remaining degrees of freedom \(c_{2h\Delta r,ij}^{k}\) for \(j = N_i-1, \dots, 1\), by solving the lower part of the linear system:
            \begin{equation}\label{eq:electrochemistry_step1-2_c2}
                AC_i^{k} - MC_i^{k-1} + J_{h,i}^{k}e_{N_i} = 0.
            \end{equation}

            \begin{remark}
                Equation~\eqref{eq:electrochemistry_step1-1_c2_scalar} constitutes \(M_2\) independent nonlinear equations indexed by \(i\). Owing to the tridiagonal structure of the matrix \(A\) in \eqref{eq:electrochemistry_step1-2_c2}, once the surface value \(c_{2h\Delta r,iN_i}^{k}\) is determined from \textbf{Step 1}, the remaining unknowns can be computed efficiently without invoking a direct or iterative linear solver. Specifically, one can use backward substitution from the \((N_i - 1)\)-th equation down to the first equation.
            \end{remark}

        \subsection{Second decoupling: Jacobian elimination}
            To solve the nonlinear system given by equations \eqref{eq:electrochemistry_step1-1_c1}-\eqref{eq:electrochemistry_step1-1_c2_scalar} (i.e., \textbf{Step 1} above), we employ Newton's method with a line search strategy:

            1. (Initialization) Set the initial guess $X^0:=\qty(c_{1 h}^{k,0}, \phi_{1 h}^{k,0}, \phi_{2 h}^{k,0}, \bar c_{2 h}^{k,0})$, typically as
            $$\phi_{1 h}^{k,0} = \phi_{1 h}^{k-1},\, \phi_{2 h}^{k,0} = \phi_{2 h}^{k-1},\, c_{1h}^{k,0} = c_{1h}^{k-1},\, \bar c^{k,0}_{2 h}=c_{2 h\Delta r}^{k-1}\qty(\cdot, R_\mathrm{s}\qty(\cdot)).$$ 
            Define the auxiliary quantities 
            $$U_h^{k,0} = \sum_{m\in \set{\text{n,p}}} U_m(\bar c_{2h}^{k,0}) \boldsymbol{1}_{\Omega_m},\quad \eta_{h}^{k,0} = \phi_{2h}^{k,0} - \phi_{1h}^{k,0}- U_h^{k,0}.$$
            Let \( rtol \) denote the relative tolerance and set the iteration index \( n = 1 \).
            
            2. (Newton Step) At each iteration \(n\), compute the Newton direction \( d^n \) by solving the linear system 
            \begin{equation}\label{eq:solver_newton_eq}
                J^n d^n + F^n = 0,
            \end{equation}
            where $J^n$ and $F^n$ are the Jacobian and residual evaluated at \( X^{n-1} \).

            3. (Line Search) Determine a step size \( \gamma^n \) using a line search algorithm, and update
            $$ X^n:=  X^{n-1} + \gamma^n d^n. $$

            4. (Convergence Check) If $\norm{\frac{X^{n} - X^{n-1}}{{X^{n-1}}}}_{l^\infty} < rtol $, terminate the iteration and set $\phi_{1 h}^{k} = \phi_{1 h}^{k,n}$, $\phi_{2 h}^{k} = \phi_{2 h}^{k,n}$, $c_{1h}^{k} = c_{1h}^{k,n}$, $\bar c_{2 h}^{k}=\bar c_{2 h}^{k,n}$. Otherwise, increment \(n\) and continue the iteration.

            The main computational burden lies in solving the linear system \eqref{eq:solver_newton_eq}.
            To define the residual vector $F$ and Jacobian matrix $J$, we arrange the equations and degrees of freedom in the following order: $c_{1 h}$, $\phi_{1 h}$, $\phi_{2 h}$ and $c_{2h\Delta r,iN_i}^{k}$, $ i=1,\dots, M_2$. This ordering allows the residual vector to be expressed as
            \begin{equation}
                F = \qty[ F_{c_1}^T\quad F_{\phi_1}^T\quad F_{\phi_2}^T\quad F_{c_2 1}\quad \cdots \quad  F_{c_2 M_2}]^T. 
            \end{equation} 
            Here, the scalar function $F_{c_2 i}$ represents the residual for the $i$-th equation in \eqref{eq:electrochemistry_step1-1_c2_scalar}, while
            the $i$-th component of the vector function $F_{c_1}$, $F_{\phi_1}$ or $F_{\phi_2}$ corresponds to the residual obtained by choosing a nodal basis function \( \varphi_i \) associated with the respective degree of freedom as the test function in the variational formulation. For example,
            \begin{equation}
                \qty(F_{c_1})_i = \int_{ \Omega} \varepsilon_1\frac{c_{1 h}^{k}-c_{1 h}^{k-1}}{\tau} \varphi_i  \, \dd x +\int_{\Omega} k_1 \nabla c_{1 h}^{k} \cdot \nabla \varphi_i \, \dd x -\int_{\Omega_2} a_1  J_h^{k} \varphi_i \, \dd x.
            \end{equation} 
            Grouping the variables into macroscopic \((c_{1 h}, \phi_{1 h}, \phi_{2 h})\) and microscopic \((\bar c_{2 h})\) components, the Jacobian matrix \( J \) can be expressed in the block form:
            \begin{equation}\label{eq:solver_newton_Jacobian}
                \begin{bmatrix}
                    \begin{array}{c:c}
                    J_\text{macro}    &  U_\text{src}    \\  \hdashline
                    V_\text{bdry}    & D_\text{micro}    \\ 
                    \end{array}
                \end{bmatrix}
                = \begin{bmatrix}
                    \begin{array}{ccc:cccc}
                    J_{c_1c_1}    & J_{c_1\phi_1}    & J_{c_1\phi_2}    & u_{c_11}    & u_{c_12}    & \cdots & u_{c_1M_2}    \\ 
                    J_{\phi_1c_1} & J_{\phi_1\phi_1} & J_{\phi_1\phi_2} & u_{\phi_11} & u_{\phi_12} & \cdots & u_{\phi_1M_2} \\ 
                    J_{\phi_2c_1} & J_{\phi_2\phi_1} & J_{\phi_2\phi_2} & u_{\phi_21} & u_{\phi_22} & \cdots & u_{\phi_2M_2} \\ \hdashline
                    v_{1c_1}      & v_{1\phi_1}      & v_{1\phi_2}      & w_{11}      & 0           & \cdots & 0      \\ 
                    v_{2c_1}      & v_{2\phi_1}      & v_{2\phi_2}      & 0           & w_{22}      & \cdots & 0      \\ 
                    \vdots        & \vdots           & \vdots           & \vdots      & \vdots      & \ddots & \vdots      \\ 
                    v_{M_2c_1}    & v_{M_2\phi_1}    & v_{M_2\phi_2}    & 0           & 0           & \cdots & w_{M_2M_2}     \\ 
                    \end{array}
                \end{bmatrix}.
            \end{equation}  
            The subscript "src" indicates coupling due to source terms in equations \eqref{eq:electrochemistry_step1-1_c1}--\eqref{eq:electrochemistry_step1-1_phi2}, while "bdry" reflects contributions from particle-boundary conditions in \eqref{eq:electrochemistry_step1-1_c2_scalar}. 
            Each row of the latter block matrix in \eqref{eq:solver_newton_Jacobian}  corresponds to the Jacobian of a subvector of $F$. 
            For example, the Jacobian of $F_{c_1}$ is the first row block:
            $$\begin{bmatrix}
                \begin{array}{ccc:ccccc}
                J_{c_1c_1}    & J_{c_1\phi_1}    & J_{c_1\phi_2}    & u_{c_11}    & u_{c_12}    & \cdots & u_{c_1M_2}
                \end{array}   
            \end{bmatrix},$$
            where 
            $J_{c_1c_1}$, $J_{c_1\phi_1}$, $J_{c_1\phi_2}$ and $u_{c_1j}$ denote the partial derivatives of $F_{c_1}$ with respect to the degrees of freedom of $c_{1h}$, $\phi_{1h}$, $\phi_{2h}$ and $c_{2h,jN_j}$, respectively.
            Similarly, the Jacobian of the \( i \)-th scalar residual \( F_{c_2 i} \) is given by:
            $$\begin{bmatrix}
                \begin{array}{ccc:ccccc}
                v_{ic_1}      & v_{i\phi_1}      & v_{i\phi_2}      & 0     & \cdots      & w_{ii}      & \cdots & 0  
                \end{array}   
            \end{bmatrix},$$
            indicating coupling with all macroscopic variables and only the \( i \)-th microscopic surface unknown.

            It is the spatial independence of $c_{2h\Delta r,iN_i}^{k,n}$ in \eqref{eq:electrochemistry_step1-1_c2_scalar} that leads to the diagonal $D_\text{micro} $. 
            This fascinating structure allows a row transformation to reduce the Jacobian to lower-triangular form
            \begin{align*}
            \begin{bmatrix} I & -U_\text{src}D_\text{micro}^{-1} \\ 0 & I \end{bmatrix} \begin{bmatrix}  J_\text{macro}    &  U_\text{src} \\ V_\text{bdry}    & D_\text{micro} \end{bmatrix}  = \begin{bmatrix} J_\text{macro} - U_\text{src}D_\text{micro}^{-1}V_\text{bdry} & 0  \\ V_\text{bdry} & D_\text{micro} \end{bmatrix}.
            \end{align*}
            Accordingly, the right-hand side is transformed as
            \[
            F = 
            \begin{bmatrix} F_1 \\ F_2 \end{bmatrix}
            \quad \rightarrow \quad
            \begin{bmatrix} F_1 - U_\text{src} D_\text{micro}^{-1} F_2 \\ F_2 \end{bmatrix}.
            \]
            We can thus first solve a smaller linear system using the Schur complement of the block $D_\text{micro}$,        
            \begin{equation}
                J/D_\text{micro}=J_\text{macro} - U_\text{src}D_\text{micro}^{-1}V_\text{bdry},
            \end{equation} 
            to obtain the Newton direction for $c_{1h}$, $\phi_{1h}$ and  $\phi_{2h}$, and then recover the direction for $\bar c_{2h}$ through direct substitution.

            Although \(D_\text{micro}^{-1}\) is trivial to compute due to its diagonal structure, it is still necessary to assess whether the associated elementary row transformations introduce significant fill-in or incur substantial computational cost.
            Fortunately, the following proposition demonstrates that the sparsity pattern of the Jacobian is preserved under this transformation.

            \begin{proposition}\label{prop:schur}
                The Schur complement $J/D_\text{micro}$ shares the same sparsity pattern as $J_\text{macro}$.
            \end{proposition}
            \begin{proof}
                Due to the diagonal structure of $D_\text{micro}$, if the entry
                $$\qty( U_\text{src}D_\text{micro}^{-1}V_\text{bdry})_{ij} =u_{ik} w_{kk}v_{kj} \neq 0$$ then there exists some $k\in \set{1,2,\dots,M_2}$ such that both $u_{ik}$ and $v_{kj}$ are nonzero.

                Let $\hat e_k \in \mathcal{T}_{2,h}$ and recall that $P_1$ elements are used to discretize $c_1$,$\phi_1$ and $\phi_2$. Suppose the $i$-th degree of freedom of the full system corresponds to the $m(i)$-th degree of freedom of variable $\xi(i) \in \set{c_1,\phi_1,\phi_2}$, denoted by $\xi(i)_{m\qty(i)}$, associated with the global nodal basis $ \varphi_{l(i)}$. Then
                $$ u_{ik} = \qty(u_{\xi(i) k})_{m(i)}= \pdv{R_{\xi(i)k}}{c_{2 h \Delta r,k N_k}}, $$
                where 
                $
                R_{\xi(i)k}= \int_{\hat e_k} a J_{m}\qty(c_{1h},c_{2 h \Delta r,k N_k},\phi_{2h}-\phi_{1h}-U_m\qty(c_{2 h \Delta r,k N_k})) \varphi_{l(i)} \, \dd x
                $, with $a\in \set{-a_1,-a_2,a_2}$. Similarly, 
                $$
                    v_{ki} =   \qty(v_{k\xi(i)})_{m(i)} = \pdv{J_{k}}{\xi(i)_{m\qty(i)}},
                $$
                where $J_{k}=\frac{\tau R_m^2\qty(A^{-1})_{N_kN_k}}{\qty|\hat e_k|F}\int_{\hat e_k} J_{m}\qty(c_{1h},c_{2 h \Delta r, k N_k},\phi_{2h}-\phi_{1h}-U_m\qty(c_{2 h \Delta r, k N_k}))  \dd x$.

                Hence, a nonzero entry $\qty( U_\text{src}D_\text{micro}^{-1}V_\text{bdry})_{ij} $ may only arise if both the $i$-th and $j$-th degrees of freedom of the full system are associated with basis functions supported on the same element $\hat e_k \in \mathcal{T}_{2,h}$. In such cases, the entry $\qty(i,j)$  of $J_\text{macro}$, namely $\qty(J_{\xi(i)\xi(j)})_{m(i)m(j)}$, already exists due to the contribution of the nonlinear term $R_{\xi(i)k}$ through the derivative $\pdv{R_{\xi(i)k}}{\xi(j)_{m(j)}}$.  

            \end{proof}
            \Cref{prop:schur} implies that nonzero entries of $U_\text{src}D_\text{micro}^{-1}V_\text{bdry}$ are element-wise localized and thus sparse. From an implementation perspective, the full Jacobian $J$ does not need to be constructed explicitly. The elimination can be performed element by element (see \Cref{alg:FEA_battery_model}) to directly form $J/D_\text{micro}$. Therefore, the computational cost of this elimination process is relatively small.
            We may thus safely conclude that the proposed decoupling strategy significantly reduces memory consumption and accelerates the solution procedure.
              
            \begin{algorithm}[htbp]  
                \caption{Jacobian elimination}
                \label{alg:FEA_battery_model} 
                \begin{algorithmic}[1]
                \State \textbf{Input:}  Original sub-Jacobian $J_\text{macro}$ 
                \For{each element $\hat e_i$ in $\mathcal{T}_{2,h}$}
                    \State Compute $v_{ic_1}$, $v_{i\phi_1}$, $v_{i\phi_2}$ and $w_{ii}$
                    \For{each variable $\xi$ in $\set{c_1,\phi_1,\phi_2}$}
                        \State Get the index set of local degrees of freedom $\set{i_1,\dots,i_{P}}$
                        \For{each index $j$ in $\set{i_1,\dots,i_{P}}$}
                            \State Compute $\qty(u_{\xi i})_j$
                            \State $\qty(J_{\xi c_1},J_{\xi \phi_1},J_{\xi \phi_2})_j=\qty(J_{\xi c_1},J_{\xi \phi_1},J_{\xi \phi_2})_j -  \frac{1}{w_{ii}}\qty(u_{\xi i})_j\qty(v_{i c_1},v_{i\phi_1},v_{i \phi_2})$ 
                        \EndFor
                    \EndFor
                \EndFor
                \end{algorithmic}
            \end{algorithm}

            \subsection{An optional nonlinear Gauss-Seidel decomposition}
                In the DFN model, the source term \eqref{eq:def_J} is governed by the Butler-Volmer equation, which takes the following form:
                \begin{equation*}
                J_m = k_m c_1^{\alpha_{a,m}}\left(c_{2, \max}-\bar c_{2}\right)^{\alpha_{a,m}} \bar c_{2}^{\alpha_{c,m}}\qty(\exp\qty(\frac{\alpha_{a,m} F}{R T_0}\eta_m) - \exp\qty(\frac{-\alpha_{c,m} F}{R T_0}\eta_m)),
                \end{equation*}
                where $\alpha_{a,m} \in(0,1)$, $\alpha_{c,m} \in(0,1)$, $k_m$, $F$, $R$, $T_0$ are positive constants. The overpotential $\eta_m$ is defined as
                \begin{equation*}
                    \eta_m = \phi_2 - \phi_1 - U_m(\bar c_2).
                \end{equation*}

                As seen in equations of \eqref{eq:DFN_dfn_strong}, the system is coupled through $J$. Since the overpotential $\eta$ appears inside exponential functions, even small variations in $\phi_1$, $\phi_2$ or $\bar c_2$ can lead to large changes in $J$.
                This sensitivity can cause considerable fluctuations during each iteration and may lead to an excessive number of outer iteration steps, when an inappropriate nonlinear Gauss-Seidel decomposition is used to solve \eqref{eq:electrochemistry_step1-1_c1}--\eqref{eq:electrochemistry_step1-1_c2_scalar}.
                In fact, numerical experiments have confirmed this observation, supporting our conjecture.

                This behavior motivates us to introduce an optional nonlinear Gauss-Seidel decomposition in \textbf{Step 1}, designed to balance speed and memory usage.  In this strategy, the variables $\phi_1$, $\phi_2$ and $\bar c_2$ are solved simultaneously, followed by an outer iteration to couple them with the remaining variable $c_1$.
        
                \textbf{Step 1-Eta-0 (Initialization)}:  Set the initial value of iteration, $X^0:=\qty(c_{1 h}^{k,0}, \phi_{1 h}^{k,0}, \phi_{2 h}^{k,0}, \bar c_{2 h}^{k,0})$, typically as
                $$\phi_{1 h}^{k,0} = \phi_{1 h}^{k-1},\, \phi_{2 h}^{k,0} = \phi_{2 h}^{k-1},\, c_{1h}^{k,0} = c_{1h}^{k-1},\, \bar c^{k,0}_{2 h}=c_{2 h\Delta r}^{k-1}\qty(\cdot, R_\mathrm{s}\qty(\cdot)).$$ 
                Define the auxiliary quantities
                $$U_h^{k,0} = \sum_{m\in \set{\text{n,p}}} U_m(\bar c_{2h}^{k,0}) \boldsymbol{1}_{\Omega_m},\quad \eta_{h}^{k,0} = \phi_{2h}^{k,0} - \phi_{1h}^{k,0}- U_h^{k,0}.$$
                Let $rtol$ be the relative tolerance and set the iteration index $n=1$. 

                \textbf{Step 1-Eta-1 (Subproblem $\qty(c_1)$)}: Find $c_{1 h}^{k,n}\in V_h^{(1)}\left(\bar \Omega\right) $, such that for all $v_h \in  V_h^{(1)}\left(\bar \Omega\right) $,
                \begin{multline}\label{eq:electrochemistry_nl_coup_c1}
                    \int_{ \Omega} \varepsilon_1\frac{c_{1 h}^{k,n}-c_{1 h}^{k-1}}{\tau} v_h  \, \dd x +\int_{\Omega} k_1 \nabla c_{1 h}^{k,n} \cdot \nabla v_h \, \dd x \\ -\int_{\Omega_2} \qty(\sum_{m\in \set{\text{n,p}}}a_1  J_{m}\qty(c_{1h}^{k,n},\bar c_{2h}^{k,n-1},\eta_{h}^{k,n-1})\boldsymbol{1}_{\Omega_m}) v_h \, \dd x=0. 
                \end{multline}
                    
                \textbf{Step 1-Eta-2 (Subproblem $\qty(\phi_1,\phi_2,\bar c_{2})$)}: Find $\phi_{1 h}^{k,n} \in  W_h\left(\bar \Omega\right)$, $\phi_{2 h}^{k,n}\in V_h^{(1)}\left(\bar \Omega_2\right)$ and $c_{2h\Delta r,iN_i}^{k,n}$, $ i=1,\dots, M_2$, such that  for all $w_h \in  W_h\left(\bar \Omega\right) $, $ v_h \in V_h^{(1)}\left(\bar \Omega_2\right) $,
                \begin{equation}\label{eq:electrochemistry_opt_step2_phi1}
                    \int_{\Omega} \kappa_{1h}^{k,n}\nabla \phi_{1h}^{k,n} \cdot \nabla w_h\, \dd x - \int_{\Omega} \kappa_{2h}^{k,n}\nabla f\qty(c_{1h}^{k,n}) \cdot \nabla w_h\,\dd x - \int_{\Omega_2}  a_2  J_h^{k,n}  w_h \,\dd x =0,
                \end{equation} 
                \begin{equation}\label{eq:electrochemistry_opt_step2_phi2} 
                    \int_{\Omega_2} \sigma \nabla \phi_{2 h}^{k,n} \cdot \nabla v_{h}\,\dd x + \int_{\Omega_2}a_2  J_h^{k,n} v_{h}\, \dd x + \int_\Gamma I^k v_{h} \,\dd x  =0, 
                \end{equation} 
                \begin{equation}\label{eq:electrochemistry_opt_step2_c2_scalar}
                    c_{2h\Delta r,iN_i}^{k,n} +  e_{N_i}^{\mathrm{T}}A^{-1}e_{N_i} J_{h,i}^{k,n} - e_{N_i}^{\mathrm{T}}A^{-1}MC_i^{k-1} = 0, \; i=1,\dots, M_2.
                \end{equation}

                \textbf{Step 1-Eta-3 (Convergence Check)}: Set $X^n:=\qty(c_{1 h}^{k,n}, \phi_{1 h}^{k,n}, \phi_{2 h}^{k,n}, \bar c_{2 h }^{k,n})$. If $\norm{\frac{X^{n} - X^{n-1}}{{X^{n-1}}}}_{l^\infty} < rtol $, then terminate the iteration and set $\phi_{1 h}^{k} = \phi_{1 h}^{k,n}$, $\phi_{2 h}^{k} = \phi_{2 h}^{k,n}$, $c_{1h}^{k} = c_{1h}^{k,n}$, $\bar c_{2 h}^{k}=\bar c_{2 h}^{k,n}$. Otherwise, set $n=n+1$ and return to $\textbf{Step 1-Eta-1}$.
            \section{Numerical experiments}
            \label{sec:experiments}

                To validate the theoretical analysis and demonstrate the superiority of the proposed solver, we present numerical simulations of the DFN model using real battery parameters from \cite{smith_solid-state_2006,timms_asymptotic_pouch_2021}. These simulations were  implemented using an in-house finite element code based on the libMesh library \citep{libMeshPaper}. The computations were carried out on the high-performance computers of the State Key Laboratory of Scientific and Engineering Computing, Chinese Academy of Sciences. 

                \begin{remark}\label{rm:exp_conservation}
                    To preserve the conservation property \eqref{eq:intro_compatibility_condition} in the fully discrete setting, we must ensure the following condition holds:
                    $$
                        \int_{\Omega_2} a_2 J_h^k \, \dd x =  \int_{\Gamma} I^k \,\dd x =0.
                    $$
                    In our galvanostatic simulations, \( I^k \) is a prescribed constant, so the condition \( \int_{\Gamma} I^k \, \dd x = 0 \) is satisfied exactly, even when numerical quadrature is used to compute the integral\( \int_\Gamma I^k v_h\, \dd x \) in \eqref{eq:DFN_fem_phi2}. For more general cases, we propose the following projection to enforce exact conservation:
                    \[
                        I^k_*= I^k - \frac{1}{|\Gamma|} \int_{\Gamma} I^k \, \dd x.
                    \]
                    Alternatively, one may employ a sufficiently accurate quadrature rule to approximate the integral reliably. 
                    Since the scheme is fully implicit, the condition \( \int_{\Omega_2} a_2J_h^k \, \dd x = 0 \) is enforced indirectly via tight residual control in the nonlinear solver, thus maintaining conservation up to a high numerical tolerance in the sense of the chosen quadrature rule. 
                \end{remark}


                \begin{remark}
                    To compute $(\phi_{1h}^k, \phi_{2h}^k) \in W_h(\bar \Omega) \times V_h^1(\bar \Omega_2)$, we adopt the fixing one-point method, which is a practical variant of the null-space method. Specifically, we first solve the system in $V_h^1(\bar \Omega) \times V_h^1(\bar \Omega_2)$ by prescribing the value of $\phi_{2h}^k$ to be zero at a chosen node. Then, we shift the resulting solution into the desired space $W_h(\bar \Omega) \times V_h^1(\bar \Omega_2)$ by subtracting the mean of $\phi_{1h}^k$ over the domain:
                    \[
                    C = -\frac{1}{|\Omega|} \int_\Omega \phi_{1h}^k \, \mathrm{d}x.
                    \]
                    This procedure avoid constructing a basis for $W_h(\bar \Omega)$. In fact, the matrix system is assembled using the standard nodal basis of $V_h^1(\bar \Omega)$, and is obtained by simply deleting the row and column associated with the fixed point from the original singular matrix, as detailed in \cite[Sec.~5.2.1]{bochev_finite_2005}.
                \end{remark}

                \subsection{Convergence validation}
                Due to computational resource limitations, complete validation of the convergence order for the P3D model is reported here while partial results for the P4D model can be found in \ref{app_sec:DFN_exp_3d}.
                \begin{figure}[htbp]
                    \centering
                    \begin{subfigure}[b]{0.49\textwidth}
                        \includegraphics[width=\textwidth]{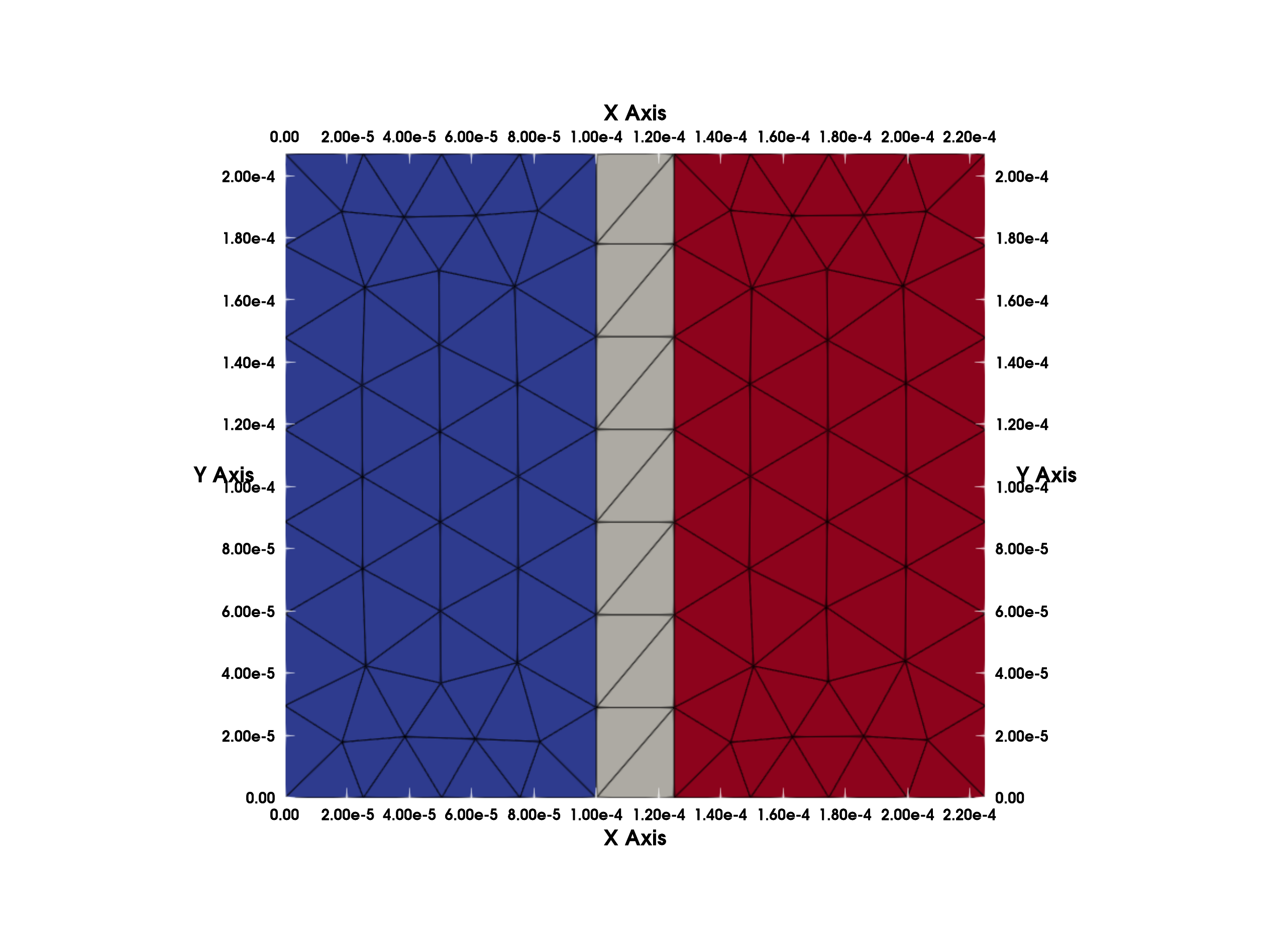}
                        \caption{}
                        \label{fig:DFN_exp_2d_mesh_r0} 
                    \end{subfigure}%
                    ~
                    \begin{subfigure}[b]{0.49\textwidth}
                        \includegraphics[width=\textwidth]{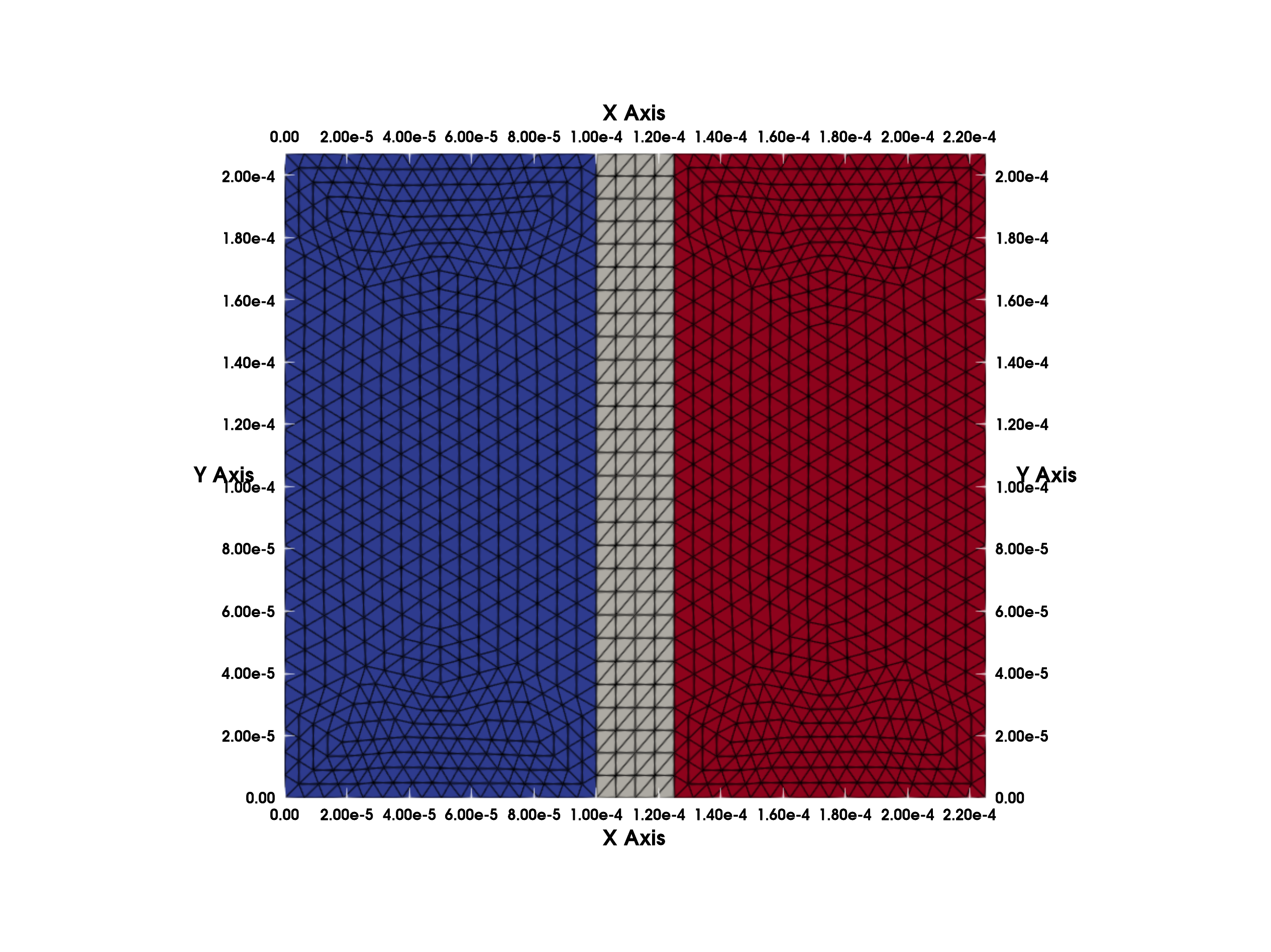}
                        \caption{} 
                        \label{fig:DFN_exp_2d_mesh_fine}
                    \end{subfigure} 
                    \caption{Spatial meshes for convergence verification. (a) Initial coarse mesh ($R_h=0$). (b) Uniformly refined mesh ($R_h=2$).}
                    \label{fig:DFN_exp_2d_mesh} 
                \end{figure}
        
                In this example, we set $\Omega_\mathrm{n} = \qty[0,100] \times \qty[0,207]$, $\Omega_\mathrm{s} = \qty[100,125] \times \qty[0,207]$, and $\Omega_\mathrm{p} = \qty[125,225] \times \qty[0,207]$ (all dimensions in $10^{-6}\mathrm{m}$). The boundary $\Gamma$ is defined as $\Gamma_\mathrm{n} \cup \Gamma_\mathrm{p}$, where $\Gamma_\mathrm{n} = \qty{0} \times \qty[0,207]$ and $\Gamma_\mathrm{p} = \qty{225} \times \qty[0,207]$. A 1C discharge rate is applied, using parameters from \cite{timms_asymptotic_pouch_2021}.
                
                The initial spatial mesh, shown in Figure~\ref{fig:DFN_exp_2d_mesh_r0}, is refined uniformly with level $R_h$. 
                The radial grid is initially uniform, with grid size $\Delta r = 1.25 \times 10^{-6}\mathrm{m}$ and refinement level $R_{\Delta r}$.
                Since the exact solution is unknown, we use the finite element solution on a very fine mesh ($R_h=5$, $R_{\Delta r} = 5$) with a sufficiently small step size $\tau_\mathrm{ref} = 0.0390625 \text{s}$  as a reference. 
                
                Figure~\ref{fig:DFN_exp_2d_err_tau} illustrates the absolute errors as a function of the time step size $\tau$ for all variables, evaluated at $T = 1.25 s$ with $R_h=5$ and $R_{\Delta r} = 5$ fixed. For fixed $R_{\Delta r} = 5$ ($R_{h} = 5$) and $\tau = \tau_\mathrm{ref}$, we refine the initial mesh from $R_h=1$ ($R_{\Delta r}=1$) to $R_h=3$ ($R_{\Delta r}=3$). The error and convergence order with respect to $h$ ($\Delta r$) at time $t_k = k\Delta t$ are summarized in Table~\ref{tab:DFN_exp_err_h} (Table~\ref{tab:DFN_exp_err_r}). The observed convergence rates agree with our theoretical analysis.
                
                \begin{figure}[htbp]
                    \centering
                    \begin{subfigure}[b]{0.49\textwidth}
                        \includegraphics[width=\textwidth]{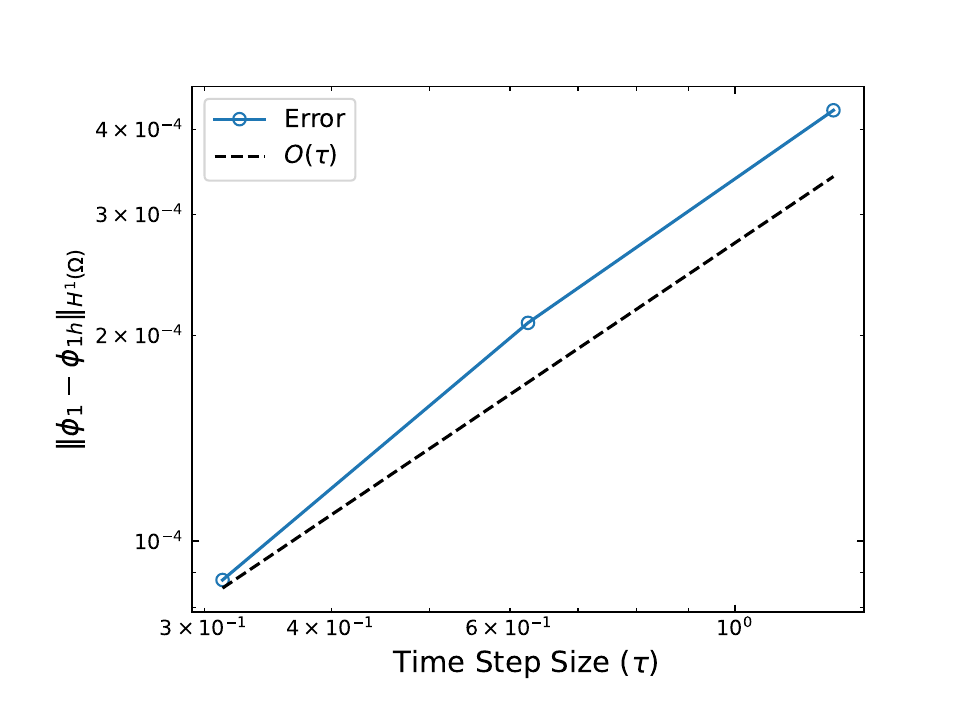}
                        \caption{} 
                        \label{fig:DFN_exp_2d_err_tau_phi1}  
                    \end{subfigure}%
                    ~
                    \begin{subfigure}[b]{0.49\textwidth}
                        \includegraphics[width=\textwidth]{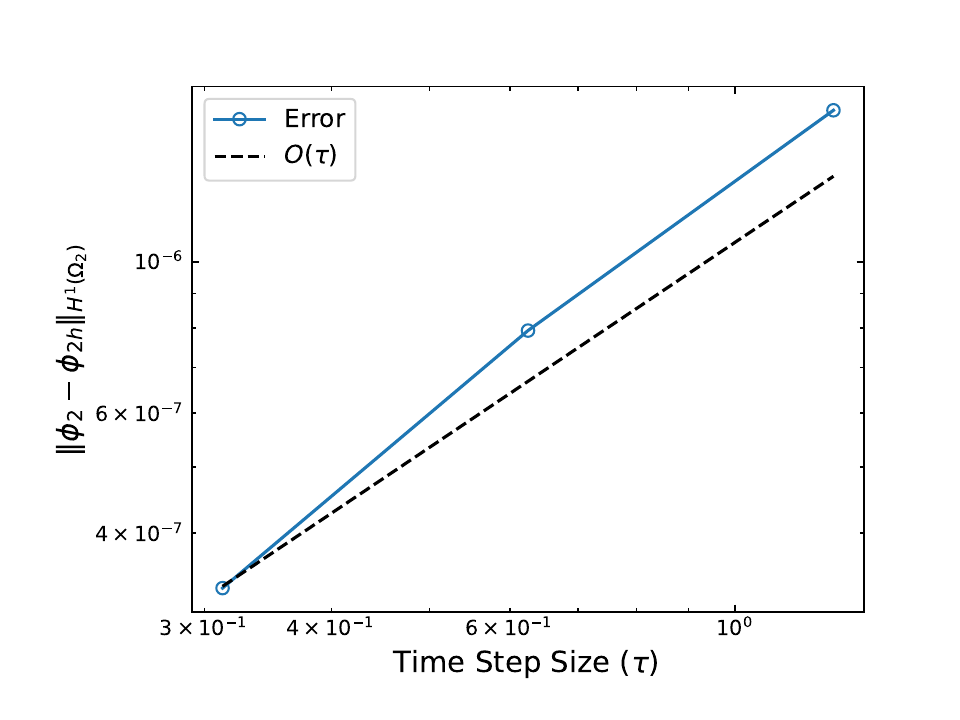}
                        \caption{} 
                        \label{fig:DFN_exp_2d_err_tau_phi2}
                    \end{subfigure} 
                    \begin{subfigure}[b]{0.49\textwidth}
                        \includegraphics[width=\textwidth]{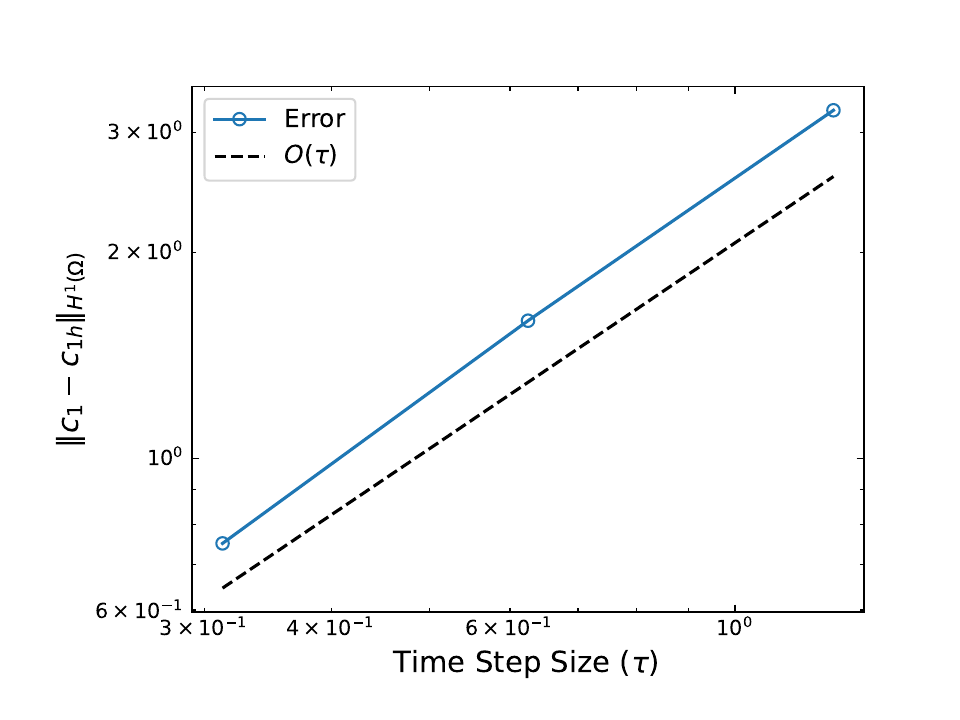}
                        \caption{}
                        \label{fig:DFN_exp_2d_err_tau_c1} 
                    \end{subfigure}%
                    ~
                    \begin{subfigure}[b]{0.49\textwidth}
                        \includegraphics[width=\textwidth]{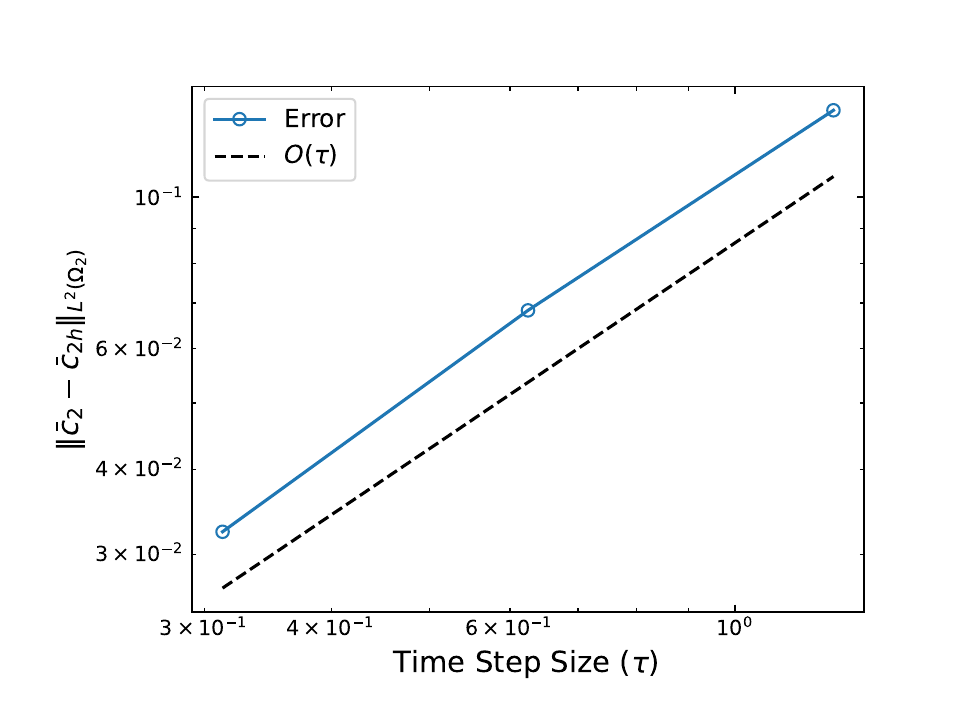}
                        \caption{} 
                        \label{fig:DFN_exp_2d_err_tau_c2_bar}
                    \end{subfigure} 
                    \begin{subfigure}[b]{0.49\textwidth}
                        \includegraphics[width=\textwidth]{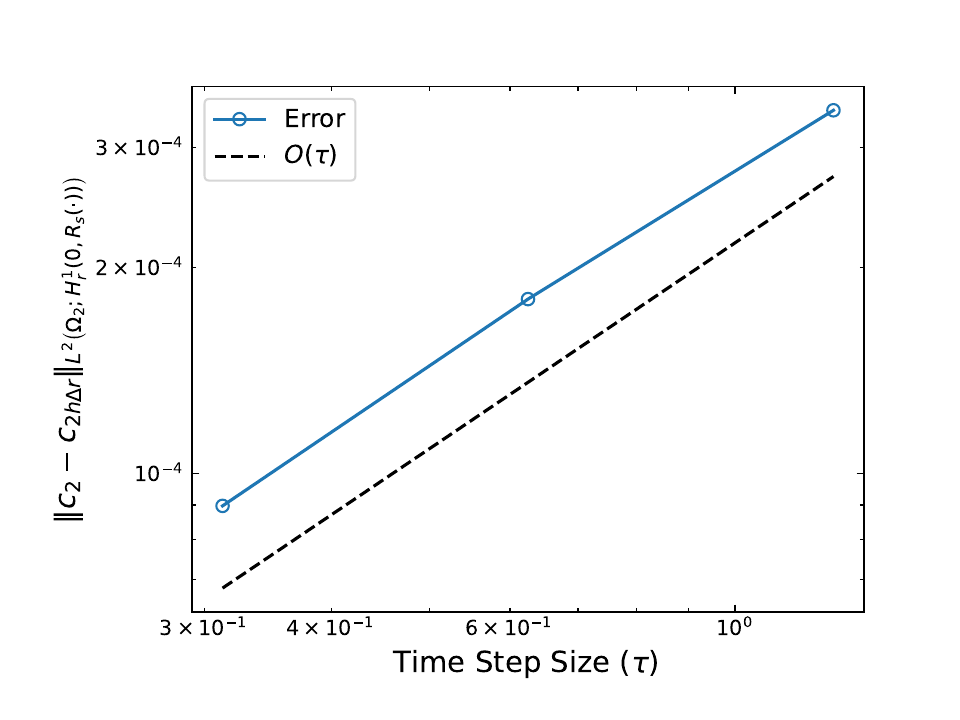}
                        \caption{}
                        \label{fig:DFN_exp_2d_err_tau_c2_H1r} 
                    \end{subfigure}%
                    ~
                    \begin{subfigure}[b]{0.49\textwidth}
                        \includegraphics[width=\textwidth]{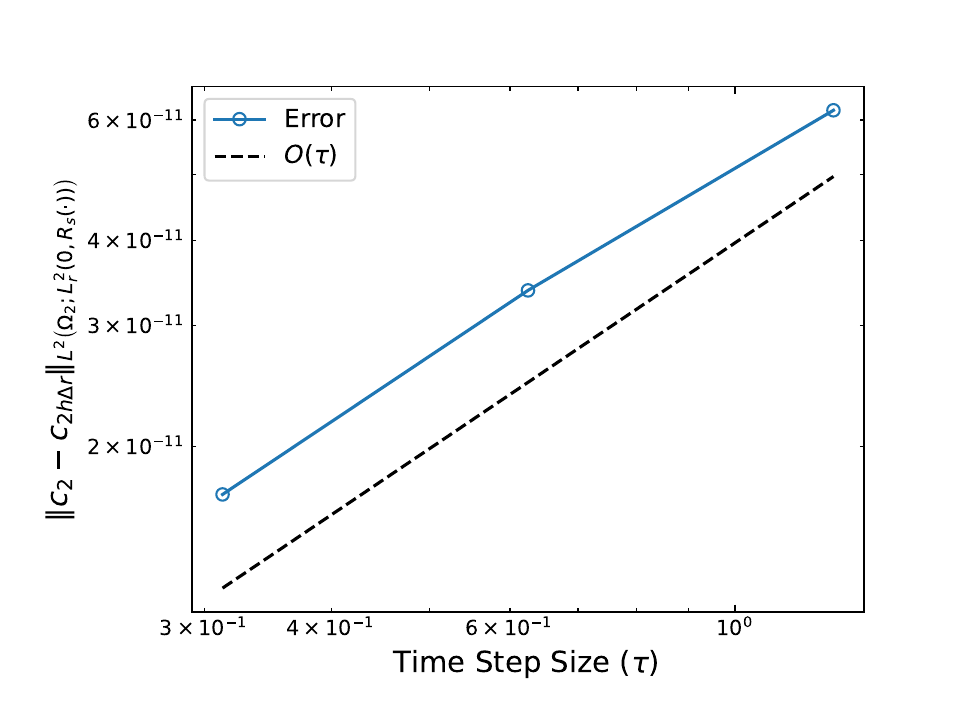}
                        \caption{} 
                        \label{fig:DFN_exp_2d_err_tau_c2_L2r}
                    \end{subfigure} 
                    \caption{Time convergence obtained at $t = 1.25\ s$. (a) Error for $\phi_1$ in the $H^1$-norm. (b) Error for $\phi_2$ in the $H^1$-norm. (c) Error for $c_1$ in the $H^1$-norm. (d) Error for $\bar c_2$ in the $L^2$-norm. (e) Error for $c_2$ in the $L^2\qty(H_r^1)$-norm. (f) Error for $c_2$ in the $L^2\qty(L_r^2)$-norm.}
                    \label{fig:DFN_exp_2d_err_tau} 
                \end{figure}

                \begin{table}[htbp]
                    \centering
                    \footnotesize
                    \caption{Error and convergence order for $h$.}\label{tab:DFN_exp_err_h}
                    \setlength{\tabcolsep}{4pt}
                    \renewcommand{\arraystretch}{1.2}
                    \begin{subtable}{0.49\textwidth}
                        \centering
                        \subcaption{$\norm{\phi_1(\cdot,t_k) - \phi_{1h}^k}_{H^1\qty(\Omega)}$}
                        \label{tab:DFN_exp_err_h_phi1_H1}
                        \begin{tabular}{ccccc}
                            \toprule
                            $k$ & $R_h=1$ & $R_h=2$ & $R_h=3$ & Order \\
                            \midrule
                            2 & 9.90E-04 & 4.92E-04 & 2.40E-04 & 1.04 \\
                            4 & 9.90E-04 & 4.90E-04 & 2.39E-04 & 1.04 \\
                            6 & 1.02E-03 & 5.06E-04 & 2.47E-04 & 1.04 \\
                            8 & 1.06E-03 & 5.25E-04 & 2.56E-04 & 1.04 \\
                            10 & 1.05E-03 & 5.19E-04 & 2.53E-04 & 1.04 \\
                            \bottomrule
                            \end{tabular}                        
                    \end{subtable}
                    \begin{subtable}{0.49\textwidth}
                        \centering
                        \subcaption{$\norm{\phi_2(\cdot,t_k) - \phi_{2h}^k}_{H^1\qty(\Omega_2)}$}
                        \label{tab:DFN_exp_err_h_phi2_H1}
                        \begin{tabular}{ccccc}
                            \toprule
                            $k$ & $R_h=1$ & $R_h=2$ & $R_h=3$ & Order \\
                            \midrule
                            2 & 3.32E-05 & 1.65E-05 & 8.06E-06 & 1.03 \\
                            4 & 3.32E-05 & 1.65E-05 & 8.06E-06 & 1.03 \\
                            6 & 3.32E-05 & 1.65E-05 & 8.06E-06 & 1.03 \\
                            8 & 3.32E-05 & 1.65E-05 & 8.06E-06 & 1.03 \\
                            10 & 3.32E-05 & 1.65E-05 & 8.06E-06 & 1.03 \\
                            \bottomrule
                        \end{tabular}                
                    \end{subtable}
                    \begin{subtable}{0.49\textwidth}
                        \centering
                        \subcaption{$\norm{c_1(\cdot,t_k) - c_{1h}^k}_{H^1\qty(\Omega)}$}
                        \label{tab:DFN_exp_err_h_c1_H1}
                        \setlength{\tabcolsep}{2.6pt}
                        \begin{tabular}{ccccc}
                            \toprule
                            $k$ & $R_h=1$ & $R_h=2$ & $R_h=3$ & Order \\
                            \midrule
                            2 & 4.24E+00 & 2.25E+00 & 1.11E+00 & 1.02 \\
                            4 & 4.70E+00 & 2.37E+00 & 1.16E+00 & 1.03 \\
                            6 & 5.70E+00 & 2.89E+00 & 1.42E+00 & 1.03 \\
                            8 & 6.80E+00 & 3.42E+00 & 1.68E+00 & 1.03 \\
                            10 & 6.83E+00 & 3.40E+00 & 1.66E+00 & 1.04 \\
                            \bottomrule
                        \end{tabular}
                    \end{subtable}
                    \begin{subtable}{0.49\textwidth}
                        \centering
                        \subcaption{$\norm{\bar c_2(\cdot,t_k) - \bar c_{2h}^k}_{L^2\qty(\Omega_2)}$}
                        \label{tab:DFN_exp_err_h_c2_surf_L2}
                        \begin{tabular}{ccccc}
                            \toprule
                            $k$ & $R_h=1$ & $R_h=2$ & $R_h=3$ & Order \\
                            \midrule
                            2 & 1.36E-02 & 6.77E-03 & 3.31E-03 & 1.04 \\
                            4 & 2.96E-03 & 1.47E-03 & 7.18E-04 & 1.04 \\
                            6 & 7.91E-03 & 3.94E-03 & 1.92E-03 & 1.03 \\
                            8 & 1.73E-02 & 8.63E-03 & 4.21E-03 & 1.04 \\
                            10 & 1.21E-02 & 5.99E-03 & 2.92E-03 & 1.03 \\
                            \bottomrule
                            \end{tabular}                        
                    \end{subtable}
                    \begin{subtable}{0.49\textwidth}
                        \centering
                        \subcaption{$\norm{c_2(\cdot,t_k) - c_{2h\Delta r}^k}_{L^2\qty(\Omega_2;H^1_r\qty(0,R_\text{s}\qty(\cdot)))}$}
                        \label{tab:DFN_exp_err_h_c2_H1r}
                        \begin{tabular}{ccccc}
                            \toprule
                            $k$ & $R_h=1$ & $R_h=2$ & $R_h=3$ & Order \\
                            \midrule
                            2 & 2.45E-05 & 1.22E-05 & 5.94E-06 & 1.03 \\
                            4 & 9.85E-06 & 4.89E-06 & 2.39E-06 & 1.04 \\
                            6 & 1.16E-05 & 5.81E-06 & 2.84E-06 & 1.03 \\
                            8 & 2.71E-05 & 1.36E-05 & 6.62E-06 & 1.04 \\
                            10 & 1.86E-05 & 9.26E-06 & 4.52E-06 & 1.03 \\
                            \bottomrule
                        \end{tabular}
                    \end{subtable}
                    \begin{subtable}{0.49\textwidth}
                        \centering
                        \subcaption{$\norm{c_2(\cdot,t_k) - c_{2h\Delta r}^k}_{L^2\qty(\Omega_2;L^2_r\qty(0,R_\text{s}\qty(\cdot)))}$}
                        \label{tab:DFN_exp_err_h_c2_L2r}
                        \begin{tabular}{ccccc}
                            \toprule
                            $k$ & $R_h=1$ & $R_h=2$ & $R_h=3$ & Order \\
                            \midrule
                            2 & 3.44E-12 & 1.71E-12 & 8.35E-13 & 1.04 \\
                            4 & 2.01E-12 & 9.97E-13 & 4.87E-13 & 1.04 \\
                            6 & 2.43E-12 & 1.21E-12 & 5.91E-13 & 1.03 \\
                            8 & 4.85E-12 & 2.42E-12 & 1.18E-12 & 1.04 \\
                            10 & 4.59E-12 & 2.28E-12 & 1.11E-12 & 1.03 \\
                            \bottomrule
                            \end{tabular}                        
                    \end{subtable}
                \end{table}
        
                \begin{table}[htbp]
                    \centering
                    \footnotesize
                    \caption{Error and convergence order for $r$.}\label{tab:DFN_exp_err_r}
                    \setlength{\tabcolsep}{4pt}
                    \renewcommand{\arraystretch}{1.2}
                    \begin{subtable}{0.49\textwidth}
                        \centering
                        \subcaption{$\norm{\phi_1(\cdot,t_k) - \phi_{1h}^k}_{H^1\qty(\Omega)}$}
                        \label{tab:DFN_exp_err_r_phi1_H1}
                        \begin{tabular}{ccccc}
                            \toprule
                            $k$ & $R_h=1$ & $R_h=2$ & $R_h=3$ & Order \\
                            \midrule
                            2 & 1.33E-05 & 3.04E-06 & 7.09E-07 & 2.10 \\
                            4 & 1.86E-05 & 4.58E-06 & 1.08E-06 & 2.08 \\
                            6 & 5.07E-06 & 1.32E-06 & 3.15E-07 & 2.06 \\
                            8 & 2.60E-06 & 6.31E-07 & 1.50E-07 & 2.07 \\
                            10 & 4.99E-06 & 1.31E-06 & 3.14E-07 & 2.06 \\
                            \bottomrule
                            \end{tabular}                        
                    \end{subtable}
                    \begin{subtable}{0.49\textwidth}
                        \centering
                        \subcaption{$\norm{\phi_2(\cdot,t_k) - \phi_{2h}^k}_{H^1\qty(\Omega_2)}$}
                        \label{tab:DFN_exp_err_r_phi2_H1}
                        \begin{tabular}{ccccc}
                            \toprule
                            $k$ & $R_h=1$ & $R_h=2$ & $R_h=3$ & Order \\
                            \midrule
                            2 & 5.52E-08 & 1.28E-08 & 2.99E-09 & 2.10 \\
                            4 & 7.75E-08 & 1.89E-08 & 4.46E-09 & 2.08 \\
                            6 & 2.10E-08 & 5.40E-09 & 1.29E-09 & 2.07 \\
                            8 & 1.01E-08 & 2.44E-09 & 5.80E-10 & 2.07 \\
                            10 & 1.89E-08 & 4.98E-09 & 1.19E-09 & 2.06 \\
                            \bottomrule
                        \end{tabular}                
                    \end{subtable}
                    \begin{subtable}{0.49\textwidth}
                        \centering
                        \subcaption{$\norm{c_1(\cdot,t_k) - c_{1h}^k}_{H^1\qty(\Omega)}$}
                        \label{tab:DFN_exp_err_r_c1_H1}
                        \begin{tabular}{ccccc}
                            \toprule
                            $k$ & $R_h=1$ & $R_h=2$ & $R_h=3$ & Order \\
                            \midrule
                            2 & 1.87E-02 & 4.34E-03 & 1.02E-03 & 2.10 \\
                            4 & 9.75E-03 & 2.49E-03 & 5.93E-04 & 2.07 \\
                            6 & 2.15E-03 & 5.88E-04 & 1.42E-04 & 2.05 \\
                            8 & 8.44E-03 & 2.03E-03 & 4.84E-04 & 2.07 \\
                            10 & 1.09E-02 & 2.75E-03 & 6.57E-04 & 2.07 \\
                            \bottomrule
                        \end{tabular}
                    \end{subtable}
                    \begin{subtable}{0.49\textwidth}
                        \centering
                        \subcaption{$\norm{\bar c_2(\cdot,t_k) - \bar c_{2h}^k}_{L^2\qty(\Omega_2)}$}
                        \label{tab:DFN_exp_err_r_c2_surf_L2}
                        \begin{tabular}{ccccc}
                            \toprule
                            $k$ & $R_h=1$ & $R_h=2$ & $R_h=3$ & Order \\
                            \midrule
                            2 & 4.20E-03 & 1.01E-03 & 2.38E-04 & 2.08 \\
                            4 & 1.94E-03 & 4.41E-04 & 1.05E-04 & 2.08 \\
                            6 & 1.38E-03 & 3.25E-04 & 7.74E-05 & 2.07 \\
                            8 & 1.08E-03 & 2.54E-04 & 6.05E-05 & 2.07 \\
                            10 & 9.64E-04 & 2.28E-04 & 5.44E-05 & 2.07 \\
                            \bottomrule
                            \end{tabular}                        
                    \end{subtable}
                    \begin{subtable}{0.49\textwidth}
                        \centering
                        \subcaption{$\norm{c_2(\cdot,t_k) - c_{2h\Delta r}^k}_{L^2\qty(\Omega_2;H^1_r\qty(0,R_\text{s}\qty(\cdot)))}$}
                        \label{tab:DFN_exp_err_r_c2_H1r}
                        \begin{tabular}{ccccc}
                            \toprule
                            $k$ & $R_h=1$ & $R_h=2$ & $R_h=3$ & Order \\
                            \midrule
                            2 & 1.72E-04 & 8.55E-05 & 4.17E-05 & 1.04 \\
                            4 & 1.33E-04 & 6.60E-05 & 3.22E-05 & 1.03 \\
                            6 & 1.09E-04 & 5.44E-05 & 2.65E-05 & 1.03 \\
                            8 & 9.85E-05 & 4.90E-05 & 2.39E-05 & 1.03 \\
                            10 & 8.94E-05 & 4.46E-05 & 2.18E-05 & 1.03 \\
                            \bottomrule
                        \end{tabular}
                    \end{subtable}
                    \begin{subtable}{0.49\textwidth}
                        \centering
                        \subcaption{$\norm{c_2(\cdot,t_k) - c_{2h\Delta r}^k}_{L^2\qty(\Omega_2;L^2_r\qty(0,R_\text{s}\qty(\cdot)))}$}
                        \label{tab:DFN_exp_err_r_c2_L2r}
                        \begin{tabular}{ccccc}
                            \toprule
                            $k$ & $R_h=1$ & $R_h=2$ & $R_h=3$ & Order \\
                            \midrule
                            2 & 2.01E-12 & 5.02E-13 & 1.23E-13 & 2.04 \\
                            4 & 1.69E-12 & 4.26E-13 & 1.03E-13 & 2.04 \\
                            6 & 1.48E-12 & 3.71E-13 & 9.00E-14 & 2.04 \\
                            8 & 1.41E-12 & 3.52E-13 & 8.51E-14 & 2.05 \\
                            10 & 1.34E-12 & 3.35E-13 & 8.10E-14 & 2.05 \\
                            \bottomrule
                            \end{tabular}                        
                    \end{subtable} 
                \end{table}
        
                \subsection{Performance comparison}\label{subsec:exp_solver}
                    In this section, we compare the performance of several solvers, including the newly proposed fully coupled solver (referred to as "2DS-FC"), the solver with the optional nonlinear Gauss-Seidel (NGS) decomposition ("2DS-Eta") and the solver only combining the first decoupling with the optional NGS decomposition ("1DS-Eta"). The 2DS-Eta solver is designed to balance solution speed and memory overhead, while 1DS-Eta serves as a baseline to evaluate the impact of Jacobian elimination. We also benchmark these solvers against existing solvers outlined in \ref{app_sec:solvers} — "GSN-Macro", "GSN-Phi", "GSN-FD" — and the fully coupled solver without any decoupling ("GSN-FC"). 
                    
                    The comparison was conducted in the P4D setting with the geometry and battery parameters provided in \ref{app_sec:DFN_exp_3d}. We considered a simulation time of $T=20$s and a uniform time step size of $\tau = 0.1$s. The grids used for discretization in the radial direction were $\set{1-\frac{1}{2^n}}_{n=1}^{9}\cup \set{0,1}$, with dimensions in $10^{-6}\mathrm{m}$. All nonlinear (sub)problems were solved to high accuracy using the Newton's method with an absolute iteration tolerance of 1e-13. The default line search algorithm in PETSc \cite{petsc-user-ref} was employed to ensure robust convergence. The linear system at each Newton step was solved using SuperLU\_DIST \cite{SuperLU-DIST}, a distributed-memory sparse direct solver for large sets of linear equations. To enhance computational efficiency, we utilized parallel computing with 252 processes.

                    We complement the review of existing solvers in Section~\ref{sec:intro} with the numerical results presented in Table~\ref{tab:electrochemistry_exp_solver_cmp}.  
                    While the GSN-FD solver is attractive due to its simplicity, it proves inefficient in terms of speed and robustness due to the strong coupling and nonlinearity of the system. 
                    Moreover, the comparison between the GSN-Phi and GSN-Macro solvers indicates that a higher degree of coupling does not necessarily result in increased efficiency. 
                    The superior speed performance of the GSN-FC solver explains its adoption in some software that advertise fast solutions. However, it quickly reaches memory limitations as the problem size increases.
                     
                    Numerical results demonstrate that the newly proposed solver 2DS-FC significantly outperforms its competitors, nearly doubling the speed of the fastest existing solver. Additionally, its maximum memory usage remains effectively bounded, similar to the GSN-Macro solver.  Thanks to the novel NGS decomposition,  the number of outer iterations is halved in the 1DS-Eta and 2DS-Eta solvers, with the reduction being contingent on a better understanding of the problem's physics and the strength of coupling between the governing equations. It is worth noting that Jacobian elimination plays a critical role in the 2DS-Eta, accelerating the 1DS-Eta by reducing the matrix order. This reduction is particularly valuable when using tetrahedral meshes, where the number of elements can be approximately five times the number of nodes. Therefore, the 2DS-Eta solver appears to be a good choice when there are stricter memory constraints.


                    \begin{table}[htbp]
                        \caption{\enspace Performance Comparison Table}
                        \label{tab:electrochemistry_exp_solver_cmp}
                        \begin{subtable}{\textwidth}
                            \caption{\enspace Parallel running time (hours)}
                            \label{tab:electrochemistry_time_mpi}
                            \centering
                            \footnotesize
                            \setlength{\tabcolsep}{1.8pt}
                            \renewcommand{\arraystretch}{1.2}
                            \begin{tabular}{ccccccccc}
                                \toprule
                                \#Nodes & \#Elems & GSN-FD & GSN-Phi & GSN-Macro & GSN-FC  & 1DS-Eta & 2DS-Eta & 2DS-FC \\
                                \midrule
                                510 & 2135 & 0.14 & 0.07 & 0.08 & 0.03 & 0.06 & 0.05 & 0.02\\ 
                                3424 & 17080 & 0.81 & 0.33 & 0.51 & 0.23 & 0.40 & 0.22 & 0.10\\ 
                                25007 & 136640 & 7.27 & 2.61 & 4.34 & 1.78 & 2.83 & 1.76 & 0.79\\ 
                                190973 & 1093120 & 95.67 & 30.67 & 56.36 & 20.79 & 36.35 & 17.74 & 10.11\\ 
                                \bottomrule
                            \end{tabular}
                        \end{subtable}
                        \begin{subtable}{\textwidth}
                            \caption{\enspace Total CPU time (hours)}
                            \label{tab:electrochemistry_time_cpu_sum}
                            \centering
                            \footnotesize
                            \setlength{\tabcolsep}{1.8pt}
                            \renewcommand{\arraystretch}{1.2}
                            \begin{tabular}{ccccccccc}
                                \toprule
                                \#Nodes & \#Elems & GSN-FD & GSN-Phi & GSN-Macro & GSN-FC  & 1DS-Eta & 2DS-Eta & 2DS-FC \\
                                \midrule
                                510 & 2135 & 34.84 & 17.22 & 34.45 & 6.99 & 15.03 & 12.76 & 4.39\\
                                3424 & 17080 & 204.41 & 81.38 & 128.11 & 57.69 & 99.50 & 54.17 & 24.06\\
                                25007 & 136640 & 1826.68 & 655.10 & 1090.04 & 446.84 & 712.10 & 441.89 & 198.71\\
                                190973 & 1093120 & 24043.68 & 7707.35 & 14162.93 & 5224.92 & 9132.94 & 4457.45 & 2540.35\\
                                \bottomrule
                            \end{tabular}
                        \end{subtable}
                        \begin{subtable}{\textwidth}
                            \caption{\enspace Average number of outer iterations}
                            \label{tab:electrochemistry_cnt_its}
                            \centering
                            \footnotesize
                            \setlength{\tabcolsep}{1.8pt}
                            \renewcommand{\arraystretch}{1.2}
                            \begin{tabular}{ccccccccc}
                                \toprule
                                \#Nodes & \#Elems & GSN-FD & GSN-Phi & GSN-Macro & GSN-FC  & 1DS-Eta & 2DS-Eta & 2DS-FC \\
                                \midrule
                                510 & 2135 & 19.06 & 10.13 & 9.80 & 1.00 & 5.33 & 5.38 & 1.00\\
                                3424 & 17080 & 29.15 & 10.22 & 9.87 & 1.00 & 4.73 & 4.73 & 1.00\\
                                25007 & 136640 & 34.29 & 10.09 & 9.81 & 1.00 & 4.64 & 4.64 & 1.00\\
                                190973 & 1093120 & 36.18 & 9.96 & 9.72 & 1.00 & 4.86 & 4.94 & 1.00\\
                                \bottomrule
                            \end{tabular}
                        \end{subtable}
                        \begin{subtable}{\textwidth}
                            \caption{\enspace Max memory usage (GB)}
                            \label{tab:electrochemistry_memory}
                            \centering
                            \footnotesize
                            \setlength{\tabcolsep}{1.8pt}
                            \renewcommand{\arraystretch}{1.2}
                            \begin{tabular}{ccccccccc}
                                \toprule
                                \#Nodes & \#Elems & GSN-FD & GSN-Phi & GSN-Macro & GSN-FC  & 1DS-Eta & 2DS-Eta & 2DS-FC \\
                                \midrule
                                510 & 2135 & 16.65 & 16.83 & 16.68 & 16.73 & 17.04 & 16.97 & 16.12\\
                                3424 & 17080 & 23.25 & 22.98 & 24.13 & 30.32 & 24.93 & 23.31 & 23.44\\
                                25007 & 136640 & 63.25 & 69.81 & 79.28 & 137.85 & 84.13 & 69.99 & 75.92\\
                                190973 & 1093120 & 390.10 & 438.00 & 532.02 & 985.57 & 552.79 & 471.34 & 532.79\\
                                \bottomrule
                            \end{tabular}
                        \end{subtable}
                    \end{table}
        
            \section{Conclusions}\label{sec:conclusions} 
        
                In this paper, we first present an error analysis of the backward Euler finite element discretization for the DFN model of lithium-ion cells. Building on the multiscale projection from  \cite{xu2024DFNsemiFEM}, we establish the optimal convergence rates in $N$-dimensions ($1\le N\le 3$) for the first time. 
                
                We then propose a novel and highly efficient solver that that reduces computational complexity through two decoupling procedures. The first decoupling reduce the dimension from $N+1$ to $N$ via local inversion, while the second reduce the order of Jacobian, exploiting the Jacobian's reducibility. To further enhance the balance between speed and memory usage, we introduce an optional nonlinear Gauss-Seidel decomposition, which adapts to the problem's physics and the extent of coupling between variables.
                
                Numerical experiments with real battery parameters validate our theoretical results and demonstrate the exceptional performance of the proposed solver. A comprehensive comparison shows that the 2DS solver is the fastest among existing solvers and exhibits robust performance, being largely insensitive to finer microscopic discretization.  Additionally, it can be easily combined with other nonlinear Gauss-Seidel decompositions for further optimization.
                
                The DFN model is the most widely used physics-based model for lithium-ion cells and a cornerstone of battery multi-physics modeling. Our theoretical analysis  ensures the reliability of the numerical scheme, while the efficient solver can be applied in fields requiring high accuracy and fast computations, such as real-time battery monitoring, estimation, or battery design optimization.
                
\appendix
\section{Function spaces and notations}\label{app_sec:func_space}

    Let $H^m(\Omega)(m \in \mathbb{N}_0)$ denote the Sobolev spaces $W^{m,2}\qty(\Omega)$ defined on a generic bounded domain $\Omega$ with the norm $\norm{\cdot}_{m,\Omega} = \norm{\cdot}_{W^{m,2}\qty(\Omega)}$. We also denote the subspace of $H^1\qty(\Omega)$, consisting of functions whose integral is zero, by $$H^1_*\qty(\Omega) := \qty{ v \in H^1\qty(\Omega):\: \int_\Omega v\, \dd x = 0}.$$
    Since multiple domains $\Omega_\mathrm{n}$, $\Omega_\mathrm{s}$ and $\Omega_\mathrm{p}$ are involved, it is necessary to further define piecewise Sobolev spaces 
    \begin{displaymath}
        H^m_\mathrm{pw}\qty(\Omega_1) =  H^m\qty(\Omega_\mathrm{n}) \cap H^m\qty(\Omega_\mathrm{s}) \cap H^m\qty(\Omega_\mathrm{p})
    \end{displaymath}
    equipped with the norm $\norm{\cdot}_{m,\Omega_1}:= \norm{\cdot}_{m,\Omega_\mathrm{n} } + \norm{\cdot}_{m,\Omega_\mathrm{s} } + \norm{\cdot}_{m,\Omega_\mathrm{p}}$,
    and 
    \begin{displaymath}
        H^m_\mathrm{pw}\qty(\Omega_2) =   H^m\qty(\Omega_\mathrm{n}) \cap H^m\qty(\Omega_\mathrm{p})
    \end{displaymath}
    with $\norm{\cdot}_{m,\Omega_2}:= \norm{\cdot}_{m,\Omega_\mathrm{n}} + \norm{\cdot}_{m,\Omega_\mathrm{p}}$.

    The natural space for radial solutions with radial coordinate $r$ is $L^2_r\qty(0,R)$, consisting of measurable functions $v$ defined on $\qty(0,R)$ such that $vr$ is $L^2$-integrable. It is obvious that $L^2_r\qty(0,R)$ is a Hilbert space and can be endowed with the norm 
    \begin{displaymath}
        \norm{v}_{L^2_r\qty(0,R)} = \qty( \int_{0}^R \qty|v(r)|^2r^2 \, \dd r )^\frac{1}{2}.
    \end{displaymath}
    We also denote by $H^m_r\qty(0,R)(m \in \mathbb{N}_0)$ the Hilbert spaces of measurable functions $v$, whose distribution derivatives belong to $L^2_r\qty(0,R)$ up to order $m$, with the norm 
    \begin{equation}
        \norm{v}_{H^m_r\qty(0,R)} = \qty( \sum_{k=0}^{m} \norm{\dv[k]{v}{r}}^2_{L^2_r\qty(0,R)})^\frac{1}{2}.
    \end{equation}
    For functions in $H^1_r\qty(0,R)$, we have the following critical trace estimation in \cite{bermejo_numerical_2021}:
    \begin{proposition}\label{prop:DFN_radial_surface}
        There exist an arbitrarily small number $\epsilon$ and a positive (possibly large) constant $C(\epsilon)$, such that for $u \in H^1_r\qty(0,R)$,
        \begin{displaymath}  
            \left|u\left(R\right)\right| 
            \leq \epsilon\left\|\frac{\partial u}{\partial r}\right\|_{L_r^2\left(0, R\right)}  + 
            C( \epsilon)\left\|u\right\|_{L_r^2\left(0, R\right)}.
        \end{displaymath}
    \end{proposition}

    Let $X$ be a Banach space. Vector-valued Lebesgue spaces $L^p\qty(\Omega;X)$, $1\le p \le \infty$, and Sobolev spaces $H^m\qty(\Omega;X)$, $m\in \mathbb{N}_0$, for a generic domain $\Omega$ are introduced. 
    It is convenient for $c_2$ with $\Omega_2$ and $R_\mathrm{s}$ in Section~\ref{sec:intro} to define 
    \begin{displaymath}
        H^p\qty(\Omega_2 ; H^q_r \qty(0,R_\mathrm{s}(\cdot))):= H^p\qty(\Omega_\mathrm{n} ; H^q_r \qty(0,R_\mathrm{n})) \cap H^p\qty(\Omega_\mathrm{p} ; H^q_r \qty(0,R_\mathrm{p})),\quad p,q \in \mathbb{N}_0,
    \end{displaymath} 
    with the norm $\norm{\cdot}_{p,\Omega_2;q,r} = \norm{\cdot}_{ H^p\qty(\Omega_\mathrm{n} ; H^q_r \qty(0,R_\mathrm{n}))} + \norm{\cdot}_{H^p\qty(\Omega_\mathrm{p} ; H^q_r \qty(0,R_\mathrm{p}))}$.
    Besides, for time-dependent variables with $\Omega=[0,T]$, notations $\qty(H^m\qty(0,T;X), \norm{\cdot}_{m;X})$ are used, and for a partition of the interval $[0, T]$, $0=t_0<t_1<\cdots<t_M=T$, we define $\norm{\cdot}_{0,k;X}:=\qty(\int_{t_{k-1}}^{t_k} \norm{\cdot}^2_X\,\dd t)^\frac{1}{2}$, $k = 1,2,\dots, M$.

    \section{Numerical error and order in 3D} \label{app_sec:DFN_exp_3d}
        In this example, we set $\Omega_\mathrm{n} = \qty[0,50] \times \qty[0,111.8] \times \qty[0,111.8]$, $\Omega_\mathrm{s} = \qty[50,75.4] \times \qty[0,111.8] \times \qty[0,111.8]$, and $\Omega_\mathrm{p} = \qty[75.4,111.8] \times \qty[0,111.8] \times \qty[0,111.8]$ (all dimensions in $10^{-6}\mathrm{m}$). The boundary $\Gamma$ is defined as $\Gamma_\mathrm{n} \cup \Gamma_\mathrm{p}$, where $\Gamma_\mathrm{n} = \qty{0} \times \qty[0,111.8] \times \qty[0,111.8]$ and $\Gamma_\mathrm{p} = \qty{111.8} \times \qty[0,111.8] \times \qty[0,111.8]$. A 5C (30A) discharge rate is applied, using parameters from \cite{smith_solid-state_2006}. 

        The initial spatial mesh is shown by Figure~\ref{fig:DFN_exp_mesh} and the initial radial grid is uniform with the grid size $\Delta r = 1.25\times 10^{-7}\mathrm{m}$. 
        The finite element solution in an extremely fine mesh ($R_h=5$ and $R_{\Delta r} = 5$) with a sufficiently small step size $\tau_\mathrm{ref} = 0.15625 \text{s}$ is taken as the reference solution.

        Fixing $R_h=5$, $R_{\Delta r} = 5$ and refining the initial time step size twice, the convergence error and order of all numerical solutions with respect to $\tau$ at $T = 2.5\ s$ are presented in Table~\ref{tab:DFN_exp_3d_err_t}.
        Fixing $R_{\Delta r} = 5$, $\tau = \tau_\mathrm{ref}$ and refining the initial spatial mesh twice, the convergence error and order with respect to $h$ at  $t_k = k\tau$ are presented in  Table~\ref{tab:DFN_exp_3d_err_h}.  We can clearly see the convergence rates $\order{\tau}$ and $\order{h}$, which agree with our analysis.
            \begin{figure}[!htbp]
                \centering
                \begin{subfigure}[b]{0.49\textwidth}
                    \includegraphics[width=\textwidth]{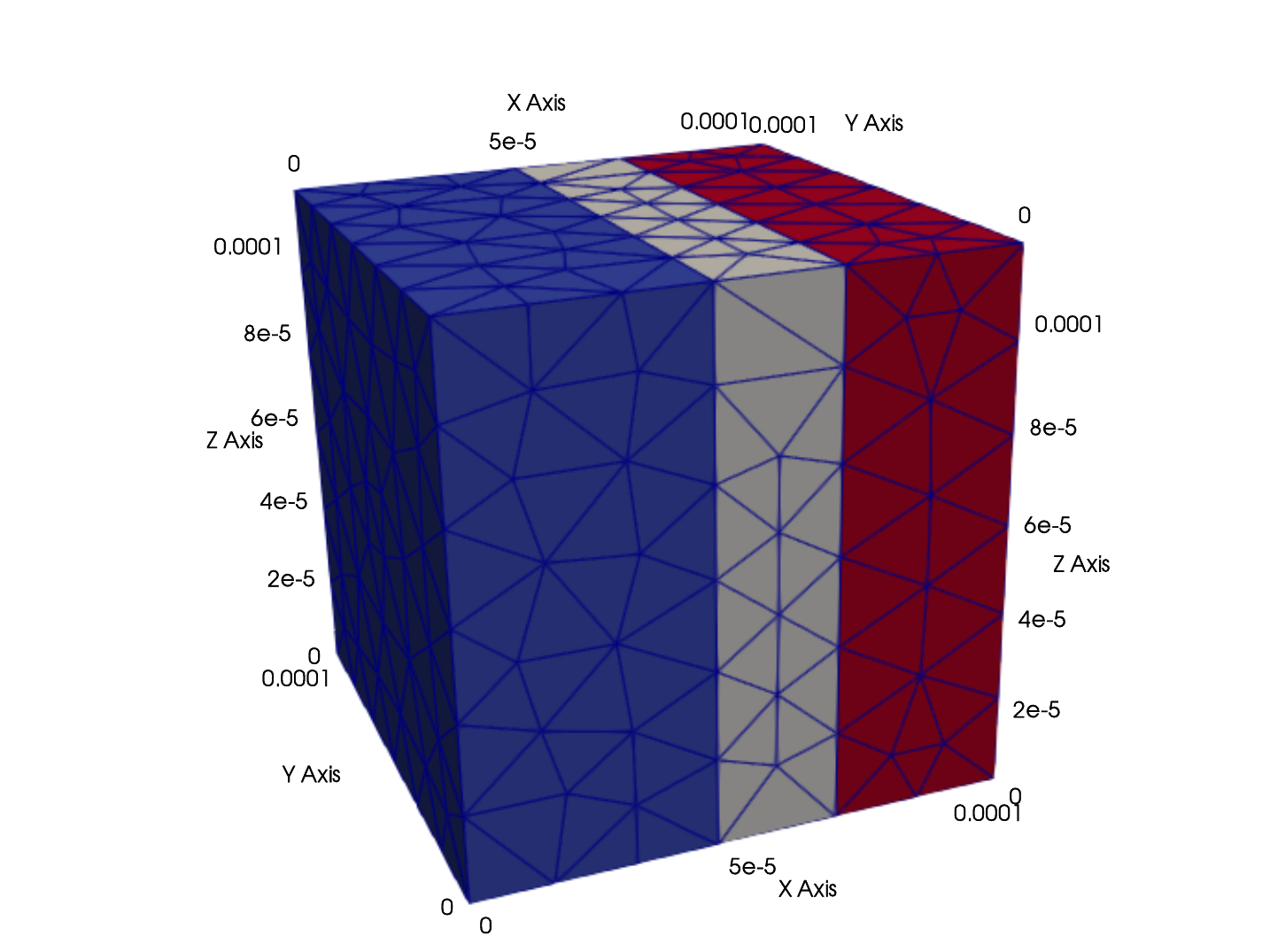}
                    \caption{}
                    \label{fig:DFN_exp_3d_mesh_r0}
                \end{subfigure}%
                ~
                \begin{subfigure}[b]{0.49\textwidth}
                    \includegraphics[width=\textwidth]{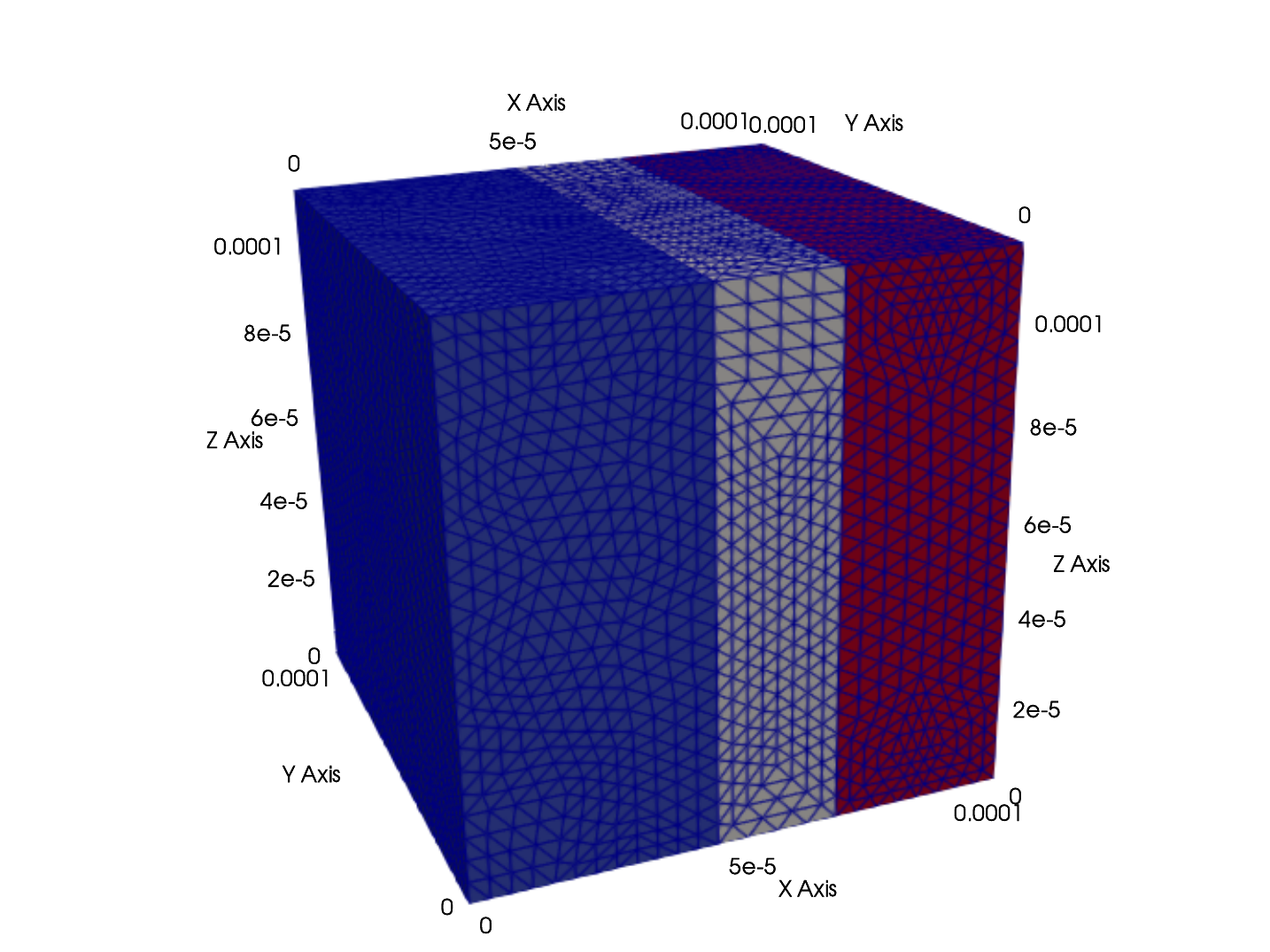}
                    \caption{}
                    \label{fig:DFN_exp_3d_mesh_r2}
                \end{subfigure} 
                \caption{Spatial meshes for convergence verification. (a) Initial coarse mesh ($R_h=0$). (b) Uniformly refined mesh ($R_h=2$).}
                \label{fig:DFN_exp_mesh}
            \end{figure}

            \begin{table}[htbp]
                \centering
                \caption{Error and convergence order for $\tau$.}
                \label{tab:DFN_exp_3d_err_t}
                \centering
                \footnotesize
                \setlength{\tabcolsep}{4pt}
                \renewcommand{\arraystretch}{1.2}
                \begin{subtable}{\textwidth}
                    \centering
                    \begin{tabular}{ccccc}
                        \toprule
                        $t_k = 2.5[s]$ & $R_{\tau}=0$ & $R_{\tau}=1$ & $R_{\tau}=2$& Order \\
                        \midrule
                        $\norm{c_1(\cdot,t_k) - c_{1h}^k}_{H^1}$ & 1.08E-01 & 5.17E-02 & 2.32E-02 & 1.16 \\
                        $\norm{\bar c_2(\cdot,t_k) - \bar c_{2h}^k}_{L^2}$ & 3.04E-01 & 1.40E-01 & 6.07E-02 & 1.21 \\
                        $\norm{\phi_1(\cdot,t_k) - \phi_{1h}^k}_{H^1}$ & 3.28E-06 & 1.44E-06 & 6.46E-07 & 1.15 \\
                        $\norm{\phi_2(\cdot,t_k) - \phi_{2h}^k}_{H^1}$ & 3.22E-08 & 2.92E-09 & 1.32E-09 & 1.15 \\
                        \bottomrule
                    \end{tabular}                    
                \end{subtable}
            \end{table}

            \begin{table}[!htbp]
                \centering
                \caption{Error and convergence order for $h$.}
                \label{tab:DFN_exp_3d_err_h}
                \centering
                \footnotesize
                \setlength{\tabcolsep}{4pt}
                \renewcommand{\arraystretch}{1.2}
                \begin{subtable}{0.49\textwidth}
                    \centering
                    \subcaption{$\norm{\phi_1(\cdot,t_k) - \phi_{1h}^k}_{H^1\qty(\Omega)}$}
                    \label{tab:DFN_exp_err_phi1_H1}
                    \begin{tabular}{ccccc}
                        \toprule
                        $k$ & $R_h=0$ & $R_h=1$ & $R_h=2$ & {Order} \\
                        \midrule
                        1   & 4.48E-06 & 2.51E-06 & 1.29E-06 & 0.96 \\
                        16  & 9.48E-06 & 4.85E-06 & 2.39E-06 & 1.02 \\
                        32  & 1.20E-05 & 6.06E-06 & 2.98E-06 & 1.02 \\
                        64  & 1.50E-05 & 7.57E-06 & 3.72E-06 & 1.02 \\
                        \bottomrule
                    \end{tabular}
                \end{subtable}
                \begin{subtable}{0.49\textwidth}
                    \centering
                    \subcaption{$\norm{\phi_2(\cdot,t_k) - \phi_{2h}^k}_{H^1\qty(\Omega_2)}$}
                    \label{tab:DFN_exp_err_phi2_H1}
                    \begin{tabular}{ccccc}
                        \toprule
                        $k$ & $R_h=0$ & $R_h=1$ & $R_h=2$ & {Order} \\
                        \midrule
                        1   & 7.40E-07 & 3.75E-07 & 1.85E-07 & 1.02 \\
                        16  & 7.40E-07 & 3.75E-07 & 1.85E-07 & 1.02 \\
                        32  & 7.40E-07 & 3.75E-07 & 1.85E-07 & 1.02 \\
                        64  & 7.40E-07 & 3.75E-07 & 1.85E-07 & 1.02 \\
                        \bottomrule
                    \end{tabular}
                \end{subtable}
                \begin{subtable}{0.49\textwidth}
                    \centering
                    \subcaption{$\norm{c_1(\cdot,t_k) - c_{1h}^k}_{H^1\qty(\Omega)}$}
                    \label{tab:DFN_exp_err_c1_H1}
                    \begin{tabular}{ccccc}
                        \toprule
                        $k$ & $R_h=0$ & $R_h=1$ & $R_h=2$ & {Order} \\
                        \midrule
                        1   & 7.29E-02 & 4.94E-02 & 2.70E-02 & 0.87 \\
                        16  & 2.46E-01 & 1.27E-01 & 6.25E-02 & 1.02 \\
                        32  & 3.26E-01 & 1.66E-01 & 8.16E-02 & 1.02 \\
                        64  & 4.26E-01 & 2.16E-01 & 1.06E-01 & 1.02 \\
                        \bottomrule
                    \end{tabular}
                \end{subtable}
                \begin{subtable}{0.49\textwidth}
                    \centering
                    \subcaption{$\norm{\bar c_2(\cdot,t_k) - \bar c_{2h}^k}_{L^2\qty(\Omega_2)}$}
                    \label{tab:DFN_exp_err_c2_surf}
                    \begin{tabular}{ccccc} 
                        \toprule
                        $k$ & $R_h=0$ & $R_h=1$ & $R_h=2$ & {Order} \\
                        \midrule
                        1 & 5.29E-01 & 2.93E-01 & 1.49E-01 & 0.97 \\
                        16 & 1.32E+00 & 6.76E-01 & 3.33E-01 & 1.02 \\
                        32 & 1.51E+00 & 7.66E-01 & 3.77E-01 & 1.02 \\
                        64 & 1.56E+00 & 7.90E-01 & 3.88E-01 & 1.02 \\
                        \bottomrule
                    \end{tabular}
                \end{subtable}
            \end{table}

        \section{Decomposition solvers in the literature}\label{app_sec:solvers} 
            Here, we provide a comprehensive list of all decomposition solvers available in the literature, to the best of our knowledge, whose performance is compared in Subsection~\ref{subsec:exp_solver}.

            Let $\left(\phi_{1 h}^{k-1}, \phi_{2 h}^{k-1}, c_{1h}^{k-1}, c_{2 h\Delta r}^{k-1}\right) $ be the fully discrete solution of the $\qty(k-1)$-th time step. We focus on solving for the solution at the $k$-th time step, $\left(\phi_{1 h}^{k}, \phi_{2 h}^{k}, c_{1h}^{k}, c_{2 h\Delta r}^{k}\right) $.
            All solvers considered in this study follow the same initialization and convergence check procedures as described below:

            (Initialization): Set the initial guess $X^0:=\qty(c_{1 h}^{k,0}, \phi_{1 h}^{k,0}, \phi_{2 h}^{k,0}, c_{2 h\Delta r}^{k,0})$, typically as
            $$\phi_{1 h}^{k,0} = \phi_{1 h}^{k-1},\, \phi_{2 h}^{k,0} = \phi_{2 h}^{k-1},\, c_{1h}^{k,0} = c_{1h}^{k-1},\, c_{2 h\Delta r}^{k,0}=c_{2 h\Delta r}^{k-1}.$$ 
            Define the auxiliary quantities
            $$\bar c^{k,0}_{2 h}=c_{2 h\Delta r}^{k,0}\qty(\cdot, R_\mathrm{s}\qty(\cdot)),\; U_h^{k,0} = \sum_{m\in \set{\text{n,p}}} U_m(\bar c_{2h}^{k,0}) \boldsymbol{1}_{\Omega_m},\; \eta_{h}^{k,0} = \phi_{2h}^{k,0} - \phi_{1h}^{k,0}- U_h^{k,0}.$$
            Let $rtol$ be the relative tolerance and set the iteration index $n=1$. 

            (Convergence Check): $X^n:=\qty(c_{1 h}^{k,n}, \phi_{1 h}^{k,n}, \phi_{2 h}^{k,n}, c_{2 h\Delta r}^{k,n})$. If $\norm{\frac{X^{n} - X^{n-1}}{{X^{n-1}}}}_{l^\infty} < rtol $, then the iteration is terminated, and we set $\phi_{1 h}^{k} = \phi_{1 h}^{k,n}$, $\phi_{2 h}^{k} = \phi_{2 h}^{k,n}$, $c_{1h}^{k} = c_{1h}^{k,n}$, $c_{2 h\Delta r}^{k}=c_{2 h\Delta r}^{k,n}$. Otherwise, set $n=n+1$, update the auxiliary quantities, and continue the iteration.

            \subsection{Macroscale-coupled solver \cite{latz_multiscale_2015}}
            
            \textbf{Step 1 (Subproblem $\qty(c_1,\phi_1,\phi_2)$)}: Find  $c_{1 h}^{k,n}\in V_h^{(1)}\left(\bar \Omega\right) $, $\phi_{1 h}^{k,n} \in  W_h\left(\bar \Omega\right)$ and $\phi_{2 h}^{k,n}\in V_h^{(1)}\left(\bar \Omega_2\right)$, such that  for all $v_{1h} \in  V_h^{(1)}\left(\bar \Omega\right) $, $w_h \in  W_h\left(\bar \Omega\right) $ and $ v_{2h} \in V_h^{(1)}\left(\bar \Omega_2\right) $, 
            \begin{multline}\label{app_eq:solvers_scale_c1}
                \int_{ \Omega} \varepsilon_1\frac{c_{1 h}^{k,n}-c_{1 h}^{k-1}}{\tau} v_{1h}  \, \dd x +\int_{\Omega} k_1 \nabla c_{1 h}^{k,n} \cdot \nabla v_{1h} \, \dd x \\ -\int_{\Omega_2} \qty(\sum_{m\in \set{\text{n,p}}}a_1  J_{m}\qty(c_{1h}^{k,n},\bar c_{2h}^{k,n-1}, \tilde \eta_{h}^{k,n})\boldsymbol{1}_{\Omega_m}) v_{1h} \, \dd x=0,
            \end{multline}
            \begin{multline}\label{app_eq:solvers_scale_phi1}
                \int_{\Omega} \kappa_{1h}^{k,n}\nabla \phi_{1h}^{k,n} \cdot \nabla w_h\, \dd x - \int_{\Omega} \kappa_{2h}^{k,n}\nabla f\qty(c_{1h}^{k,n}) \cdot \nabla w_h\,\dd x \\ - \int_{\Omega_2}  \qty(\sum_{m\in \set{\text{n,p}}}a_2  J_{m}\qty(c_{1h}^{k,n},\bar c_{2h}^{k,n-1},\tilde \eta_{h}^{k,n})\boldsymbol{1}_{\Omega_m})  w_h \,\dd x =0,
            \end{multline} 
            \begin{multline}\label{app_eq:solvers_scale_phi2} 
                \int_{\Omega_2} \sigma \nabla \phi_{2 h}^{k,n} \cdot \nabla v_{2h}\,\dd x + \int_{\Omega_2}\qty(\sum_{m\in \set{\text{n,p}}}a_2  J_{m}\qty(c_{1h}^{k,n},\bar c_{2h}^{k,n-1},\tilde \eta_{h}^{k,n})\boldsymbol{1}_{\Omega_m}) v_{2h}\, \dd x \\ 
                + \int_\Gamma I^k v_{2h} \,\dd x  =0, 
            \end{multline} 
            where $\kappa_{ih}^{k,n} = \sum_{m\in \set{\text{n,s,p}}}\kappa_{im}\qty(c_{1h}^{k,n})\boldsymbol{1}_{\Omega_m}$, $i=1,2$ and $\tilde \eta_{h}^{k,n} = \phi_{2h}^{k,n} - \phi_{1h}^{k,n} - U_h^{k,m-1}$.

            \textbf{Step 2  (Subproblem $c_2$)}: Find $c_{2 h \Delta r}^{k,n}\in V_{h \Delta r}\left(\bar \Omega_ {2 r}\right)$, such that for all  $v_{h \Delta r} \in V_{h \Delta r}\left(\bar \Omega_ {2 r}\right)$,
            \begin{multline}\label{app_eq:solvers_scale_c2}
                \int_{\Omega_2} \int_0^{R_\mathrm{s}(x)} \frac{c_{2 h \Delta r}^{k,n} - c_{2 h \Delta r}^{k-1}}{\tau} v_{h \Delta r} r^ 2 \dd r \dd x+\int_{\Omega_2} \int_0^{R_\mathrm{s}(x)} k_2 \frac{\partial c_{2 h \Delta r}^{k,n}}{\partial r} \frac{\partial v_{h \Delta r}}{\partial r} r^2 \dd r \dd x \\ +\int_{\Omega_2} \qty(\sum_{m\in \set{\text{n,p}}} \frac{R_\mathrm{s}^2}{F}  J_{m}\qty(c_{1h}^{k,n},\bar c_{2h}^{k,n},\eta_{h}^{k,n})\boldsymbol{1}_{\Omega_m}) v_{h \Delta r}\left(x, R_\mathrm{s}( x)\right) \dd x = 0,
            \end{multline} 
            where $\eta_{h}^{k,n} = \phi_{2h}^{k,n} - \phi_{1h}^{k,n} - U_h^{k,n}$, and $U_h^{k,n} = \sum_{m\in \set{\text{n,p}}} U_m(\bar c_{2h}^{k,n}) \boldsymbol{1}_{\Omega_m}$.

            \subsection{Potential-coupled solver \cite{bermejo_implicit-explicit_2019}}
            
                \textbf{Step 1 (Subproblem $c_1$)}: Find $c_{1 h}^{k,n}\in V_h^{(1)}\left(\bar \Omega\right) $, such that for all $v_h \in  V_h^{(1)}\left(\bar \Omega\right) $,
                \begin{multline}\label{app_eq:solvers_potential_c1}
                    \int_{ \Omega}\varepsilon_1\frac{c_{1 h}^{k,n}-c_{1 h}^{k-1}}{\tau} v_h  \, \dd x +\int_{\Omega} k_1 \nabla c_{1 h}^{k,n} \cdot \nabla v_h \, \dd x \\ -\int_{\Omega_2} \qty(\sum_{m\in \set{\text{n,p}}}a_1  J_{m}\qty(c_{1h}^{k,n},\bar c_{2h}^{k,n-1},\eta_{h}^{k,n-1})\boldsymbol{1}_{\Omega_m}) v_h \, \dd x=0. 
                \end{multline}
                    
                \textbf{Step 2 (Subproblem $\qty(\phi_1,\phi_2)$)}: Find $\phi_{1 h}^{k,n} \in  W_h\left(\bar \Omega\right)$, $\phi_{2 h}^{k,n}\in V_h^{(1)}\left(\bar \Omega_2\right)$, such that  for all $w_h \in  W_h\left(\bar \Omega\right) $, $ v_h \in V_h^{(1)}\left(\bar \Omega_2\right) $,
                \begin{multline}\label{app_eq:solvers_potential_phi1}
                    \int_{\Omega} \kappa_{1h}^{k,n}\nabla \phi_{1h}^{k,n} \cdot \nabla w_h\, \dd x - \int_{\Omega} \kappa_{2h}^{k,n}\nabla f\qty(c_{1h}^{k,n}) \cdot \nabla w_h\,\dd x \\ - \int_{\Omega_2}  \qty(\sum_{m\in \set{\text{n,p}}}a_2  J_{m}\qty(c_{1h}^{k,n},\bar c_{2h}^{k,n-1},\tilde \eta_{h}^{k,n})\boldsymbol{1}_{\Omega_m})  w_h \,\dd x =0,
                \end{multline} 
                \begin{multline}\label{app_eq:solvers_potential_phi2}
                    \int_{\Omega_2} \sigma \nabla \phi_{2 h}^{k,n} \cdot \nabla v_{h}\,\dd x + \int_{\Omega_2}\qty(\sum_{m\in \set{\text{n,p}}}a_2  J_{m}\qty(c_{1h}^{k,n},\bar c_{2h}^{k,n-1},\tilde \eta_{h}^{k,n})\boldsymbol{1}_{\Omega_m})  v_h\, \dd x \\ 
                    + \int_\Gamma I^k v_{h} \,\dd x  =0, 
                \end{multline} 
                where $\kappa_{ih}^{k,n} = \sum_{m\in \set{\text{n,s,p}}}\kappa_{im}\qty(c_{1h}^{k,n})\boldsymbol{1}_{\Omega_m}$, $i=1,2$, $\tilde \eta_{h}^{k,n} = \phi_{2h}^{k,n} - \phi_{1h}^{k,n} - U_h^{k,n-1}$.
    
                \textbf{Step 3  (Subproblem $c_2$)}: Find $c_{2 h \Delta r}^{k,n}\in V_{h \Delta r}\left(\bar \Omega_ {2 r}\right)$, such that for all  $v_{h \Delta r} \in V_{h \Delta r}\left(\bar \Omega_ {2 r}\right)$,
                \begin{multline}\label{app_eq:solvers_potential_c2}
                    \int_{\Omega_2} \int_0^{R_\mathrm{s}(x)} \frac{c_{2 h \Delta r}^{k,n} - c_{2 h \Delta r}^{k-1}}{\tau} v_{h \Delta r} r^ 2 \dd r \dd x+\int_{\Omega_2} \int_0^{R_\mathrm{s}(x)} k_2 \frac{\partial c_{2 h \Delta r}^{k,n}}{\partial r} \frac{\partial v_{h \Delta r}}{\partial r} r^2 \dd r \dd x \\ +\int_{\Omega_2} \qty(\sum_{m\in \set{\text{n,p}}} \frac{R_\mathrm{s}^2}{F}  J_{m}\qty(c_{1h}^{k,n},\bar c_{2h}^{k,n},\eta_{h}^{k,n})\boldsymbol{1}_{\Omega_m}) v_{h \Delta r}\left(x, R_\mathrm{s}( x)\right) \dd x = 0,
                \end{multline} 
                where $\eta_{h}^{k,n} = \phi_{2h}^{k,n} - \phi_{1h}^{k,n} - U_h^{k,n}$, and $U_h^{k,n} = \sum_{m\in \set{\text{n,p}}} U_m(\bar c_{2h}^{k,n}) \boldsymbol{1}_{\Omega_m}$.

            \subsection{Fully decoupled solver \cite{kim_robust_2023}}

            \textbf{Step 1  (Subproblem $c_2$)}: Find $c_{2 h \Delta r}^{k,n}\in V_{h \Delta r}\left(\bar \Omega_ {2 r}\right)$, such that for all  $v_{h \Delta r} \in V_{h \Delta r}\left(\bar \Omega_ {2 r}\right)$,
            \begin{multline*}
                \int_{\Omega_2} \int_0^{R_\mathrm{s}(x)} \frac{c_{2 h \Delta r}^{k,n} - c_{2 h \Delta r}^{k-1}}{\tau} v_{h \Delta r} r^ 2 \dd r \dd x+\int_{\Omega_2} \int_0^{R_\mathrm{s}(x)} k_2 \frac{\partial c_{2 h \Delta r}^{k,n}}{\partial r} \frac{\partial v_{h \Delta r}}{\partial r} r^2 \dd r \dd x \\ +\int_{\Omega_2} \qty(\sum_{m\in \set{\text{n,p}}} \frac{R_\mathrm{s}^2}{F}  J_{m}\qty(c_{1h}^{k,n-1},\bar c_{2h}^{k,n}, \tilde \eta_{h}^{k,n})\boldsymbol{1}_{\Omega_m}) v_{h \Delta r}\left(x, R_\mathrm{s}( x)\right) \dd x = 0,
            \end{multline*} 
            where $\tilde \eta_{h}^{k,n} = \phi_{2h}^{k,n-1} - \phi_{1h}^{k,n-1} - U_h^{k,n}$, and $U_h^{k,n} = \sum_{m\in \set{\text{n,p}}} U_m(\bar c_{2h}^{k,n}) \boldsymbol{1}_{\Omega_m}$ .

            \textbf{Step 2 (Subproblem $\phi_2$)}: Find $\phi_{2 h}^{k,n}\in V_h^{(1)}\left(\bar \Omega_2\right)$, such that  for all $ v_h \in V_h^{(1)}\left(\bar \Omega_2\right) $,
                \begin{multline*}
                    \int_{\Omega_2} \sigma \nabla \phi_{2 h}^{k,n} \cdot \nabla v_{h}\,\dd x + \int_{\Omega_2}\qty(\sum_{m\in \set{\text{n,p}}}a_2  J_{m}\qty(c_{1h}^{k,n},\bar c_{2h}^{k,n-1},\tilde{\tilde{ \eta}}_{h}^{k,n})\boldsymbol{1}_{\Omega_m})  v_h\, \dd x \\ 
                    + \int_\Gamma I^k v_{h} \,\dd x  =0, 
                \end{multline*} 
                where $\tilde{\tilde{ \eta}}_{h}^{k,n} = \phi_{2h}^{k,n} - \phi_{1h}^{k,n-1} - U_h^{k,n}$. 

            \textbf{Step 3 (Subproblem $\phi_1$)}: Find $\phi_{1 h}^{k,n} \in  V_h^{(1)}\left(\bar \Omega\right)$, such that  for all $w_h \in  V_h^{(1)}\left(\bar \Omega\right)$,
            \begin{multline}\label{app_eq:solvers_decouple_phi1}
                \int_{\Omega} \kappa_{1h}^{k,n-1}\nabla \phi_{1h}^{k,n} \cdot \nabla w_h\, \dd x - \int_{\Omega} \kappa_{2h}^{k,n-1}\nabla f\qty(c_{1h}^{k,n-1}) \cdot \nabla w_h\,\dd x \\ - \int_{\Omega_2}  \qty(\sum_{m\in \set{\text{n,p}}}a_2  J_{m}\qty(c_{1h}^{k,n-1},\bar c_{2h}^{k,n},{{\eta}}_{h}^{k,n})\boldsymbol{1}_{\Omega_m})  w_h \,\dd x =0,
            \end{multline} 
            where ${{\eta}}_{h}^{k,n} = \phi_{2h}^{k,n} - \phi_{1h}^{k,n} - U_h^{k,n}$.

            \textbf{Step 4 (Subproblem $c_1$)}: Find $c_{1 h}^{k,n}\in V_h^{(1)}\left(\bar \Omega\right) $, such that for all $v_h \in  V_h^{(1)}\left(\bar \Omega\right) $,
                \begin{multline*}\label{app_eq:solvers_decouple_c1}
                    \int_{ \Omega} \varepsilon_1\frac{c_{1 h}^{k,n}-c_{1 h}^{k-1}}{\tau} v_h  \, \dd x +\int_{\Omega} k_1 \nabla c_{1 h}^{k,n} \cdot \nabla v_h \, \dd x \\ -\int_{\Omega_2} \qty(\sum_{m\in \set{\text{n,p}}}a_1  J_{m}\qty(c_{1h}^{k,n},\bar c_{2h}^{k,n},\eta_{h}^{k,n})\boldsymbol{1}_{\Omega_m}) v_h \, \dd x=0.
                \end{multline*}  

 \bibliographystyle{elsarticle-num-names} 
 \bibliography{bib/PDE.bib, bib/FEM.bib, bib/review.bib, bib/DFN.bib, bib/DFN_heat.bib, bib/DFN_mechanics.bib,bib/solver.bib}
 


\end{document}